\documentclass{nhm-prep}
\usepackage{amsmath,paralist,amssymb,amsfonts,epsfig,float,cite,comment,mathtools,amsthm}
\usepackage[colorlinks=true]{hyperref}
\hypersetup{urlcolor=blue, citecolor=red}

\usepackage{bbm,bm}
\usepackage{multirow}	
\usepackage{enumerate}
\usepackage{graphicx}
\usepackage{stmaryrd}

\usepackage[table]{xcolor}
\hypersetup{
    colorlinks,
    linkcolor={red!50!black},
    citecolor={blue!50!black},
    urlcolor={blue!80!black}
}

\usepackage[a4paper,tmargin=2.7cm,bmargin=2.4cm, rmargin=2.cm,lmargin=2.cm]{geometry}

\usepackage{booktabs, multirow}
\usepackage{enumitem}

\definecolor{lightblue}{rgb}{0.22,0.45,0.70}
\definecolor{lightgreen}{rgb}{0.22,0.50,0.25}

\newcommand\R{\mathbb{R}}
\newcommand\bbS{\mathbb{S}}
\newcommand\bbM{\mathbb{M}}

\newcommand\cC{\mathcal{C}}

\newcommand\cT{\mathcal{T}}
\newcommand\cP{\mathcal{P}}

\newcommand{\cF}{\mathcal{F}}
\newcommand\mt{\mathtt{t}}
\newcommand\bn{\boldsymbol{n}}
\newcommand\bF{\boldsymbol{f}}
\newcommand\bG{\boldsymbol{g}}
\newcommand\bu{\boldsymbol{u}}
\newcommand\bv{\boldsymbol{v}}
\newcommand\bw{\boldsymbol{w}}
\newcommand\be{\boldsymbol{e}}

\newcommand\bx{\boldsymbol{x}}
\newcommand\bI{\boldsymbol{I}}
\newcommand\bPi{\boldsymbol{\Pi}}
\newcommand\beps{\boldsymbol{\varepsilon}}
\newcommand\bsig{\boldsymbol{\sigma}}
\newcommand\btau{\boldsymbol{\tau}}

\newcommand\bnabla{\boldsymbol{\nabla}}

\newcommand\tD{\mathtt{D}}
\newcommand\tr{\mathop{\mathrm{tr}}\nolimits}

\newcommand\bdiv{\mathop{\mathbf{div}}\nolimits}
\newcommand\sdiv{\mathop{\mathrm{div}}\nolimits}
\newcommand{\jump}[1]{\llbracket #1 \rrbracket}
\newcommand{\mean}[1]{\left\{\kern-1.ex\left\{ #1 \right\}\kern-1.ex\right\}}
\newcommand{\vertiii}[1]{{\left\vert\kern-0.25ex\left\vert\kern-0.25ex\left\vert #1 
    \right\vert\kern-0.25ex\right\vert\kern-0.25ex\right\vert}}

\DeclarePairedDelimiter\norm{\lVert}{\rVert}
\DeclarePairedDelimiter{\inner}{(}{)}
\DeclarePairedDelimiter{\set}{\{}{\}}
\DeclarePairedDelimiter{\dual}{\langle}{\rangle}
\newtheorem{remark}{Remark}[section]
\newtheorem{lemma}{Lemma}[section]
\newtheorem{theorem}{Theorem}[section]
\newtheorem{prop}{Proposition}[section]
\newtheorem{corollary}{Corollary}[section]

\numberwithin{equation}{section}
\numberwithin{figure}{section}
\numberwithin{table}{section}

\xdefinecolor{mygrey1}{rgb}{0.65,0.65,0.6375}
\xdefinecolor{mygrey2}{rgb}{0.9,0.9,0.9}
\newcolumntype{g}{ >{\columncolor{mygrey1}} c }
\newcolumntype{m}{ >{\columncolor{mygrey2}} c }
\def\currentvolume{X}

    \def\ppages{1--17}
     \def\DOI{yyy/zzx}

\allowdisplaybreaks

\title[Pure--stress DG formulation of Brinkman equations with strong symmetry]{A new DG method for {a} pure--stress formulation of the Brinkman problem with strong symmetry$^*$}
\thanks{$^*$\textit{Funding}: This research was  supported by Spain's Ministry of Economy Project PID2020-116287GB-I00, by the Monash Mathematics Research Fund S05802-3951284, and by the Australian Research Council through the Discovery Project grant DP220103160.}

\author[S. Meddahi, R. Ruiz-Baier]{}
\email{salim@uniovi.es}
\email{ricardo.ruizbaier@monash.edu}

\date{\today}
\begin{document}
\maketitle

\centerline{\scshape Salim Meddahi}
{\footnotesize
 \centerline{Facultad de Ciencias, Universidad de Oviedo}
                   \centerline{Federico Garc\'ia Lorca,  18, 33007-Oviedo, Spain}
}

\medskip
\centerline{\scshape Ricardo Ruiz-Baier}
{\footnotesize
\centerline{School of Mathematics, Monash University}
\centerline{9 Rainforest Walk, Clayton, Victoria 3800, Australia; and}
\centerline{Universidad Adventista de Chile, Casilla 7-D Chill\'an, Chile}
}

\medskip

\begin{abstract}
A strongly symmetric stress approximation is proposed for the Brinkman equations with mixed boundary conditions. The resulting formulation solves for the Cauchy stress using a symmetric interior penalty discontinuous Galerkin method. Pressure and velocity are readily post-processed from stress, and a second post-process is shown to produce exactly divergence-free discrete velocities. {We demonstrate the stability of the method with respect to a DG-energy norm and obtain error estimates that are explicit with respect to the  coefficients of the problem. We derive optimal rates of convergence for the stress and for the post-processed variables.  Moreover, under appropriate assumptions on the mesh, we prove  optimal $L^2$-error estimates for the stress.} Finally,  we provide numerical examples in 2D and 3D.  
\end{abstract}

\medskip

\medskip



\section{Introduction} 

\subsection*{Scope and related work} 
Brinkman equations constitute one of the simplest homogenized models for viscous fluid flow in highly heterogeneous porous media. The analysis of the underlying PDE system exhibits challenges due to the presence of two parameters: viscosity and permeability. Likewise, the design and analysis of numerical methods that maintain robustness with respect to the very relevant cases of low effective viscosity (Darcy-like regime) vs very large permeability (Stokes-like regime), is still an issue of high interest.

For classical velocity--pressure formulations, a number of contributions are available to deal with the inf-sup compatibility of velocity and pressure spaces irrespective of the regime (Darcy-like or Stokes-like). See, e.g., \cite{gn12,hong16b,kan,konno11,li19}. 
On the other hand, mixed finite element formulations for Brinkman equations (that is, using other fields apart from the classical velocity--pressure pair) include pseudostress--based methods \cite{GabrielAntonio,howell16} (see also VEM counterparts in \cite{cgs17,gms18}, and \cite{Qian} for the DG case with strong imposition of pseudostress symmetry, closer to the present contribution), vorticity--velocity--pressure schemes in augmented and non-augmented form \cite{alvarez16,anaya15,anaya16,anaya15b,vassilevski14}. There, the analysis of continuous and discrete problems depends either on the Babu\v{s}ka--Brezzi theory, or on a generalized inf-sup argument.

The method advanced in this paper is based on a pure--stress formulation, obtained from taking the deviatoric part of the stress and eliminating pressure, and using the momentum equation to write velocity as a function of the external and internal forces (forcing plus the divergence of the Cauchy stress). We consider the case of mixed boundary conditions, and use a symmetric interior penalty DG discretization. The space for discrete stresses incorporates symmetry strongly, as in \cite{Qian} (see also \cite{wang} for elasticity and \cite{mr22} for viscoelasticity). 
At the discrete level, the velocity and the pressure are easily recovered in terms of the discrete Cauchy stress through post-processing. While the pressure post-processing is derivative-free, the usual post-process of velocity through the momentum balance entails taking the discrete divergence of the approximate stress, therefore leading to accuracy loss. As a remedy we propose a more involved reconstruction, starting from the initial post-processed velocity and solving an additional problem in mixed form using $H(\sdiv)$-conforming approximations. The newly computed velocity is exactly divergence-free and we prove that it enjoys the same convergence order as the stress approximation.    

Other advantages of the present formulation include the physically accurate imposition of outflow boundary conditions fixing the normal traces of the full Cauchy stress (which is not straightforward to incorporate with vorticity- or pseudostress-based methods), having symmetric and positive definite linear systems after discretization, and straightforwardly handling heterogeneous and anisotropic permeability distributions.  In this regard, the analysis of the formulation uses a DG-energy norm chosen in such a way that the error estimates are  independent of the jumps of the  permeability. This implies a similar property concerning the estimates for the velocity approximation. However, for the pressure we can only achieve stability and {optimal rates of} convergence in terms of the mesh parameter.

\medskip 
\noindent\textbf{Outline.} The contents of the paper have been laid out in the following manner. The remainder of this Section lists useful notation regarding tensors, and it recalls the definition of Sobolev spaces, inner products, and integration by parts. In Section~\ref{sec:model} we state the Brinkman problem, specify assumptions on the model coefficients, and derive a weak formulation written only in terms of stress. Section~\ref{sec:prelim} gathers the preliminary concepts needed for the construction and analysis of the discrete problem. Here we include mesh properties, recall useful trace and inverse inequalities, define suitable projectors, and establish approximation properties in conveniently chosen norms. The definition of the stress-based DG method and the proof its consistency and optimal convergence {in the stress variable} is presented in Section~\ref{sec:dg}. The velocity and pressure post-processing, together with the corresponding error estimates, are given in Section~\ref{sec:postpr}. Next, {in Section~\ref{sec:L2} we derive} optimal error bounds for the stress in the $L^2$-norm, and Section~\ref{sec:results} contains a collection of numerical examples that confirm experimentally the convergence of the method, and also showcase its application into typical flow problems in 2D and 3D. 

\medskip 
\noindent\textbf{Notational preliminaries.} 
We denote the space of real matrices of order $d\times d$ by $\bbM$ and let $\bbS:= \set{\btau\in \bbM;\ \btau = \btau^{\mt} } $  be the subspace of symmetric matrices, where $\btau^{\mt}:=(\tau_{ji})$ stands for the transpose of $\btau = (\tau_{ij})$. The component-wise inner product of two matrices $\bsig, \,\btau \in\bbM$ is defined by $\bsig:\btau:= \sum_{i,j}\sigma_{ij}\tau_{ij}$. We also introduce the deviatoric part $\btau^{\tD}:=\btau-\frac{1}{d}\left(\tr\btau\right) \bI$ of a tensor $\btau$ , where $\tr\btau:=\sum_{i=1}^d\tau_{ii}$ and $\bI$ stands here for the identity in $\bbM$.

Let $\Omega$ be a polyhedral Lipschitz bounded domain of $\R^d$ $(d=2,3)$, with boundary $\partial \Omega$. Along this paper we convene to apply all differential operators row-wise.  Hence, given a tensorial function $\bsig:\Omega\to \bbM$ and a vector field $\bu:\Omega\to \R^d$, we set the divergence $\bdiv \bsig:\Omega \to \R^d$, the   gradient $\bnabla \bu:\Omega \to \bbM$, and the linearized strain tensor $\beps(\bu) : \Omega \to \bbS$ as
\[
(\bdiv \bsig)_i := \sum_j   \partial_j \sigma_{ij} \,, \quad (\bnabla \bu)_{ij} := \partial_j u_i\,,
\quad\hbox{and}\quad \beps(\bu) := \frac{1}{2}\left[\bnabla\bu+(\bnabla\bu)^{\mt}\right].
\]
 
 For $s\in \R$, $H^s(\Omega,E)$ stands for the usual Hilbertian Sobolev space of functions with domain $\Omega$ and values in E, where $E$ is either $\R$, $\R^d$ or $\bbM$. In the case $E=\R$ we simply write $H^s(\Omega)$. The norm of $H^s(\Omega,E)$ is denoted $\norm{\cdot}_{s,\Omega}$ and the corresponding semi-norm $|\cdot|_{s,\Omega}$, indistinctly for $E=\R,\R^d,\bbM$. We use the convention  $H^0(\Omega, E):=L^2(\Omega,E)$ and let $(\cdot, \cdot)$ be the inner product in $L^2(\Omega, E)$, for $E=\R,\R^d,\bbM$, namely,
\[
  (\bu, \bv):=\int_\Omega \bu\cdot\bv,\ \forall \bu,\bv\in L^2(\Omega,\R^d),\quad  (\bsig, \btau):=\int_\Omega \bsig:\btau,\ \forall \bsig, \btau\in L^2(\Omega,\bbM). 
\]
 
 The space $H(\sdiv, \Omega)$ stands for the vector fields $\bv\in L^2(\Omega, \R^d)$ satisfying $\sdiv \bv\in L^2(\Omega)$. We denote the corresponding norm $\norm{\bv}^2_{H(\sdiv,\Omega)}:=\norm{\bv}_{0,\Omega}^2+\norm{\sdiv\bv}^2_{0,\Omega}$. Similarly, the space of tensors in $L^2(\Omega, \bbS)$ with divergence in $L^2(\Omega, \R^d)$ is denoted $H(\bdiv, \Omega, \bbS)$. We maintain the same notation $\norm{\btau}^2_{H(\bdiv,\Omega)}:=\norm{\btau}_{0,\Omega}^2+\norm{\bdiv\btau}^2_{0,\Omega}$ for the corresponding norm. Let $\bn$ be the outward unit normal vector to {$\partial \Omega$}. The Green formula
\[
(\btau, \beps(\bv)) + (\bdiv \btau, \bv) = \int_{\partial \Omega} \btau\bn\cdot \bv\qquad  \forall \bv \in H^1(\Omega,\R^d),
\] 
can be used to extend the normal trace operator $ \cC^\infty(\overline \Omega, \bbS)\ni \btau \to (\btau|_{\partial \Omega})\bn$ to a linear continuous mapping $(\cdot|_{\partial \Omega})\bn:\, H(\bdiv, \Omega, \bbS) \to H^{-\frac{1}{2}}(\partial \Omega, \R^d)$, where $H^{-\frac{1}{2}}(\partial \Omega, \R^d)$ is the dual of $H^{\frac{1}{2}}(\partial \Omega, \R^d)$.

Throughout this paper, we shall use the letter $C$ to denote a generic positive constant independent of the mesh size  $h$ and the physical parameters $\kappa$ {and $\mu$}, that may stand for different values at its different occurrences. Moreover, given any positive expressions $X$ and $Y$ depending on $h$,  $\kappa$, {and $\mu$},  the notation $X \,\lesssim\, Y$  means that $X \,\le\, C\, Y$.

\section{The pure-stress formulation of the Brinkman problem}\label{sec:model}
Let $\Omega\subset \R^d$ be a bounded and connected Lipschitz domain with boundary $\Gamma:=\partial \Omega$. Our purpose is to solve the  Brinkman model  
\begin{subequations}
\begin{align}
		\bsig & = 2\mu \beps(\bu) - p \bI \quad \text{in $\Omega$}, \label{beq1}
	\\[1ex]
	\bu  &= \frac{\kappa}{\mu} (\bF + \bdiv \bsig) \quad \text{in $\Omega$}, \label{beq2}
	\\[1ex]
	\sdiv \bu &= 0 \quad \text{in $\Omega$}, \label{beq3}
\end{align}\end{subequations}
that describes the flow of a fluid with dynamic viscosity $\mu>0$, with velocity field $\bu:\Omega \to \R^d$ and pressure $p:\Omega \to \R$, in a porous medium  characterized by a permeability coefficient $\kappa\in L^\infty(\Omega)$. The volume force is $\bF\in L^2(\Omega, \R^d)$ is given and we assume that 
\[
  0< \kappa_-  \leq  \kappa(\bx) \leq \kappa_+ \quad \text{a.e. in $\Omega$}.
\]
We assume a no-slip boundary condition $\bu = \boldsymbol{0}$  on  a subset $\Gamma_D\subset\Gamma := \partial \Omega$ of positive surface measure and impose the normal stress boundary condition $\bsig\bn = \boldsymbol{0}$ on  its complement $\Gamma_N:= \Gamma \setminus \Gamma_D$, where $\bn$ represents the exterior unit normal vector on $\Gamma$. In the case $\Gamma_D = \Gamma$ we impose the zero mean value restriction $\inner*{ p, 1} = 0$ on the pressure to enforce uniqueness. 

We want to impose the stress tensor $\bsig$ as a primary variable. To this end, we write the deviatoric part of \eqref{beq1}  and use equation \eqref{beq2} to eliminate $p$ and $\bu$, respectively, and end up with the boundary value problem
\begin{subequations}
\begin{align}
	\frac{1}{2}\bsig^\tD & = \beps\Big({\kappa}(\bdiv \bsig + \bF) \Big)  \quad \text{in $\Omega$}, \label{d1}
	\\[1ex]
	\frac{\kappa}{\mu}(\bdiv \bsig + \bF) &= \mathbf 0 \quad \text{on $\Gamma_D$}, \label{d2}
	\\[1ex]
	\bsig \bn &= \mathbf 0 \quad \text{on $\Gamma_N$}. \label{d3}
\end{align}\end{subequations}
In the case of a subset $\Gamma_N$ with positive surface measure, the essential boundary condition \eqref{d3} requires the introduction of the closed subspace of $H(\bdiv, \Omega, \bbS)$ given by 
\[
H_N(\bdiv, \Omega, \bbS) := \set*{\btau\in H(\bdiv, \Omega, \bbS); \quad 
	\dual*{\btau\bn,\bv}_{\Gamma}= 0 
	\quad \text{$\forall\bv\in H^{\frac{1}{2}}(\partial\Omega,\R^d)$,\, $\bv|_{\Gamma_D} = \mathbf{0}$}},
\]
where $\dual*{\cdot, \cdot}_\Gamma$ holds for the duality pairing between $H^{\frac{1}{2}}(\Gamma,\R^d)$ and $H^{-\frac{1}{2}}(\Gamma,\R^d)$. Hence, the energy space is given by 
\[
  X:= \begin{cases}
  	H_N(\bdiv, \Omega, \bbS) & \text{if $\Gamma_D\neq \Gamma$},\\
  	H(\bdiv, \Omega, \bbS) & \text{if $\Gamma_D = \Gamma$}.
  \end{cases}
\]

 Testing \eqref{d1} with $\btau\in X$ and integrating by parts yields
\[
  \tfrac{1}{2}\inner*{ \bsig^\tD, \btau^\tD } = \inner*{  \bnabla\big( \kappa(\bdiv \bsig + \bF) \big), \btau } = - \inner*{   {\kappa}(\bdiv \bsig + \bF) , \bdiv \btau }.
\]
 This leads us to propose the following variational formulation of the problem: Find $\bsig\in X$ such that 
\begin{equation}\label{mvf}
a(\bsig, \btau) + \big(\kappa \bdiv \bsig, \bdiv\btau \big) = 
  - \big(\kappa \bF , \bdiv\btau \big), \quad \forall \btau \in X.
\end{equation}
where
\[
  a(\bsig, \btau):= \tfrac{1}{2}\inner*{ \bsig^\tD, \btau^\tD }  + \theta\,  \inner*{\tr \bsig, 1}\inner*{\tr \btau, 1},
\]
with $\theta = 1$ if a no-slip boundary condition is imposed everywhere on $\Gamma$ (namely, if $\Gamma_D = \Gamma$) and $\theta = 0$ otherwise. We point out that, in the case $\theta = 1$, testing problem \eqref{mvf} with $\btau =\bI$ gives $\inner*{\tr(\bsig),1}_\Omega = 0$, which corresponds to the zero mean value restriction on the pressure. We are then opting for a variational insertion of  this condition  in order to free the energy space $X= H(\bdiv, \Omega, \bbS)$ from  this constraint.   

\begin{theorem}
	The variational problem \eqref{mvf} admits a unique solution and there exists a constant $C>0$, depending $\Omega$ and $\kappa$, such that 
	\[
  \norm*{\bsig}_{H(\bdiv, \Omega)} \leq C \norm{\bF}_{0,\Omega}.
\]

\end{theorem}
\begin{proof}
Let us first notice that, by virtue of 
\begin{equation*}
	\norm*{\btau}^2_{0,\Omega} = \norm*{\btau^\tD}^2_{0,\Omega} + \tfrac{1}{d} \norm*{\tr \btau}^2_{0,\Omega}\quad \text{and} \quad \norm*{\tr \btau}^2_{0,\Omega} \leq d \, \norm*{\btau}^2_{0,\Omega},
\end{equation*}
the bilinear form defining the variational problem \eqref{mvf} is bounded:
\[
 \left| a(\bsig, \btau) + \big(\kappa \bdiv \bsig, \bdiv\btau \big)\right| \leq \max\{\tfrac{1}{2} + \theta d|\Omega|, \kappa_+\} \norm{\bsig}_{H(\bdiv,\Omega)} \, \norm{\btau}_{H(\bdiv,\Omega)}
\quad \forall \bsig, \btau \in X.
\]
Moreover, as a consequence  of the Poincaré--Friedrichs inequalities  (see, e.g., \cite[Proposition 9.1.1]{BoffiBrezziFortinBook})
\begin{equation}\label{ellip}
\norm*{\btau^\tD}^2_{0,\Omega} +\norm*{\bdiv \btau}_{0,\Omega}^2 \geq \alpha \norm*{\btau}^2_{0,\Omega},\quad \forall \btau \in H(\bdiv, \Omega, \bbS), \quad \inner*{\tr \btau , 1} = 0,
\end{equation}
and
(see \cite[Lemma 2.4]{GabrielAntonio})
\begin{equation}\label{ellipM}
\norm*{\btau^\tD}^2_{0,\Omega} +\norm*{\bdiv \btau}_{0,\Omega}^2 \geq \alpha \norm*{\btau}^2_{H(\bdiv,\Omega)},\quad \forall \btau \in H_N(\bdiv, \Omega, \bbS),	
\end{equation}
the bilinear form is also coercive on $X$ and the well-posedness of \eqref{mvf} is a consequence of  Lax--Milgram Lemma. Indeed, if $\Gamma_N$ has positive surface measure (which corresponds to $\theta = 0$) the coercivity of the bilinear form follows directly from  \eqref{ellipM}.  In the case $\Gamma_D = \Gamma$ (and $\theta = 1$), we can take advantage of the  $L^2(\Omega,\bbM)$-orthogonal decomposition $\btau = \btau_0 + \frac{1}{d|\Omega|} \inner*{\tr\btau, 1} \bI$ and the properties   $\btau^\tD = (\btau_0)^\tD$, $\bdiv \btau_0 = \bdiv \btau$ and $\inner*{\tr \btau_0 , 1} = 0$ to deduce the coercivity from \eqref{ellip} as follows: 
	\begin{align}\label{3.7}
	\begin{split}
	\norm*{\btau}^2_{H(\bdiv,\Omega)} &= \norm*{\btau_0}^2_{0,\Omega} + \frac{1}{d|\Omega|} \inner*{\tr\btau, 1}^2 + \norm*{\bdiv \btau }^2_{0,\Omega} \leq \frac{1}{\alpha} \Big( 
	\norm*{(\btau_0)^\tD}^2_{0,\Omega} +  \norm*{\bdiv \btau_0 }^2_{0,\Omega}\Big) 
	\\
	&\qquad  + \frac{1}{d|\Omega|} \inner*{\tr\btau, 1}^2 + \norm*{\bdiv \btau }^2_{0,\Omega}
	\leq  \beta \left( a(\btau, \btau) + \norm*{\kappa^{\frac{1}{2}}\bdiv \btau }^2_{0,\Omega}\right),
	\end{split}
	\end{align}
	for all $\btau \in H(\bdiv, \Omega, \bbM)$, with $\beta = \max\left\{(1 + \frac{1}{\alpha})\frac{1}{\kappa_-}, \frac{2}{\alpha}, \frac{1}{d|\Omega|} \right\}$.
\end{proof}

\begin{remark}
	Once the stress tensor $\bsig$ is known, the remaining variables can be recovered from  
	\begin{equation}\label{eq:post-cont}
		\bu = \frac{\kappa}{\mu}\big(\bdiv \bsig + \bF \big) \quad \text{and} \quad  p = - \tfrac{1}{d} \tr \bsig.
	\end{equation}
	Moreover, we point out that testing \eqref{mvf} with $\btau = \phi \bI$, with $\phi:\Omega \to \R$ smooth and compactly supported in $\Omega$ we readily deduce the incompressibility condition $\sdiv \bu = 0$ in $\Omega$. 
\end{remark}

\section{Auxiliary results concerning discretization}\label{sec:prelim}

From now on, we assume that there exists a polygonal/polyhedral disjoint partition $\big\{\Omega_j,\ j= 1,\ldots,J\big\}$ of  $\bar \Omega$  such that $\kappa|_{\Omega_j}:= \kappa_j$ {constant},  for all $j=1,\ldots,J$.

We consider  a sequence $\{\mathcal{T}_h\}_h$ of shape regular meshes that subdivide the domain $\bar \Omega$ into  triangles/tetrahedra $K$ of diameter $h_K$. The parameter $h:= \max_{K\in \cT_h} \{h_K\}$ represents the mesh size of $\cT_h$.  We assume that $\mathcal{T}_h$ is aligned with the partition $\bar\Omega = \cup_{j= 1}^J \bar{\Omega}_j$ and that $\cT_h(\Omega_j) := \set*{K\in \cT_h;\ K\subset \Omega_j }$ is a shape regular mesh of $\bar\Omega_j$ for all $j=1,\cdots, J$ and  all $h$.

For all $s\geq 0$, we consider the broken Sobolev space    
\[
  H^s(\cup_j\Omega_j) := \set*{ v\in L^2(\Omega);\ v|_{\Omega_j}\in H^s(\Omega_j),\ \forall j =1,\ldots,J },
\]
corresponding to the partition $\bar\Omega = \cup_{j= 1}^J \bar{\Omega}_j$. 
Its vectorial and tensorial versions are denoted $H^s(\cup_j\Omega_j,\R^d)$ and $H^s(\cup_j\Omega_j,\bbM)$, respectively. Likewise, the broken Sobolev space with respect to the subdivision of $\bar \Omega$ into $\cT_h$ is  
\[
 H^s(\cT_h,E):=
 \set*{\bv \in L^2(\Omega, E): \quad \bv|_K\in H^s(K, E)\quad \forall K\in \cT_h},\quad \text{for $E \in \set{ \R, \R^d, \bbM}$}. 
\]
For each $\bv:=\set{\bv_K}\in H^s(\cT_h,\R^d)$ and $\btau:= \set{\btau_K}\in H^s(\cT_h,\bbM)$ the components $\bv_K$ and $\btau_K$  represent the restrictions $\bv|_K$ and $\btau|_K$. When no confusion arises, the restrictions of these functions will be written without any subscript.

Hereafter, given an integer $m\geq 0$ and a domain $D\subset \mathbb{R}^d$, $\cP_m(D)$ denotes the space of polynomials of degree at most $m$ on $D$. We introduce the space   
\[
 \cP_m(\cT_h) := 
 \set*{ v\in L^2(\Omega): \ v|_K \in \cP_m(K),\ \forall K\in \cT_h },
 \]
 of piecewise polynomial functions relatively to $\cT_h$. We also consider the space $\cP_m(\cT_h,E)$ of functions with values in $E$ and entries in $\cP_m(\cT_h)$, where $E$ is either $\R^d$, $\bbM$ or $\bbS$. 
 
Let us introduce now notations related to DG approximations of $H(\text{div})$-type spaces. We say that a closed subset $F\subset \overline{\Omega}$ is an interior edge/face if $F$ has a positive $(d-1)$-dimensional measure and if there are distinct elements $K$ and $K'$ such that $F =\bar K\cap \bar K'$. A closed subset $F\subset \overline{\Omega}$ is a boundary edge/face if there exists $K\in \cT_h$ such that $F$ is an edge/face of $K$ and $F =  \bar K\cap \Gamma$. We consider the set $\cF_h^0$ of interior edges/faces, the set $\cF_h^\partial$ of boundary edges/faces and let $\cF(K):= \set{F\in \cF_h;\quad F\subset \partial K}$ be the set of edges/faces composing the boundary of $K$.

We assume that the boundary mesh $\cF_h^\partial$ is compatible with the partition $\partial \Omega = \Gamma_D \cup \Gamma_N$ in the sense that, if 
$
\cF_h^D = \set*{F\in \cF_h^\partial:\, F\subset \Gamma_D}$  and $\cF_h^N = \set*{F\in \cF_h^\partial:\, F\subset \Gamma_N},
$
then $\Gamma_D = \cup_{F\in \cF_h^D} F$ and $\Gamma_N = \cup_{F\in \cF_h^N} F$. We denote   
\[
  \cF_h := \cF_h^0\cup \cF_h^\partial\qquad \text{and} \qquad \cF^*_h:= \cF_h^{0} \cup \cF_h^{N},
\]
and for all $K\in \cT_h$. Obviously, in the case $\Gamma_D = \Gamma$ we have that $\cF^*_h = \cF^0_h$. Finally, we introduce the set $\cF(K):= \set{F\in \cF_h;\quad F\subset \partial K}$ of edges/faces composing the boundary of $K$.

 We will need the space given on the skeletons of the triangulations $\cT_h$  by $L^2(\cF^*_h):= \bigoplus_{F\in \cF^*_h} L^2(F)$. Its vector valued version is denoted $L^2(\cF^*_h,\R^d)$. Here again, the components $\bv_F$ of $\bv := \set{\bv_F}\in L^2(\cF^*_h,\R^d)$  coincide with the restrictions $\bv|_F$.  We endow $L^2(\cF^*_h,\R^d)$ with the inner product 
\[
(\bu, \bv)_{\cF^*_h} := \sum_{F\in \cF^*_h} \int_F \bu_F\cdot \bv_F\quad \forall \bu,\bv\in L^2(\cF^*_h,\R^d),
\]
and denote the corresponding norm $\norm*{\bv}^2_{0,\cF^*_h}:= (\bv,\bv)_{\cF^*_h}$. From now on, $h_\cF\in L^2(\cF^*_h)$ is the piecewise constant function defined by $h_\cF|_F := h_F$ for all $F \in \cF^*_h$ with $h_F$ denoting the diameter of edge/face $F$. By virtue of our hypotheses on $\kappa$ and on the triangulation $\cT_h$, we may consider that $\kappa$ is an element of $\cP_0(\cT_h)$ and denote $\kappa_K:= \kappa|_K$ for all $K\in \cT_h$. Moreover, we introduce $\gamma_\cF\in L^2(\cF^*_h)$ defined by $\gamma_F := \min\{ \kappa_K^{-1}, \kappa_{K'}^{-1} \}$ if $K\cap K' = F$ and $\gamma_F := \kappa_K^{-1}$ if $F\cap K \in \cF^N_h$. 

Given  $\bv\in H^s(\cT_h,\R^d)$ and $\btau\in H^s(\cT_h,\bbM)$, with $s>\frac{1}{2}$, we define averages $\mean{\bv}\in L^2(\cF^*_h,\R^d)$ and jumps $\jump{\btau}\in L^2(\cF^*_h,\R^d)$ by
\[
 \mean{\bv}_F := (\bv_K + \bv_{K'})/2 \quad \text{and} \quad \jump{\btau}_F := 
 \btau_K \bn_K + \btau_{K'}\bn_{K'} 
 \quad \forall F \in \cF(K)\cap \cF(K'),
\]
with the conventions 
\[
 \mean{\bv}_F := \bv_K  \quad \text{and} \quad \jump{\btau}_F := 
 \btau_K \bn_K  
 \quad \forall F \in \cF(K),\,\, F\in \cF_h^\partial,
\]
where $\bn_K$ is the outward unit normal vector to $\partial K$.

For any $k\geq 1$, we 
let $\mathcal{X}^k(h) :=X + \cP_{k}(\cT_h,\bbS)$.  Given $\btau \in \cP_{k}(\cT_h,\bbS)$, we define $\bdiv_h \btau \in  L^2(\Omega,\R^d)$ by $\bdiv_h \btau|_{K} := \bdiv \btau_K$ for all $K\in \cT_h$ and endow $\mathcal{X}^k(h)$ with the norm
\begin{equation}\label{eq:triple-norm}
 \vertiii{\btau}^2 := a(\btau, \btau) + \norm*{\kappa^{\frac{1}{2}}\bdiv_h  \btau}^2_{0,\Omega} + \norm*{\gamma_\cF^{-\frac{1}{2}} h_{\cF}^{-\frac{1}{2}} \jump{\btau}}^2_{0,\cF^*_h}.
\end{equation}
Under the condition $\bdiv_h \btau \in H^s(\cT_h,\R^d)$ ($s>\frac{1}{2}$),  we also introduce
\[
  \vertiii{\btau}^2_* := \norm*{\btau}^2_{0,\Omega} + \norm*{\kappa^{\frac{1}{2}}\bdiv_h  \btau}^2_{0,\Omega} + \norm*{\gamma_\cF^{\frac{1}{2}} h_F^{\frac{1}{2}} \mean{\kappa\bdiv_h \btau}}^2_{0,\cF^*_h}.
\]

We end this section by recalling  technical results needed for the convergence analysis of problem \eqref{mvf}. We begin with the following well-known the multiplicative trace inequality , see for example \cite{DiPietroErn}.
\begin{lemma}\label{card} 
There exists a constant $C>0$ independent of $h$ such that 
\begin{equation}\label{multiplicativetrace}
		h_K^{\frac{1}{2}}\norm{v}_{0,\partial K} \leq C \big( \norm{v}_{0,K} + h_K \norm*{\nabla v}_{0,K} \big),  
	\end{equation}
	for all $v\in H^1(K)$ and all $K\in \cT_h$. 
\end{lemma}
It is easy to deduce from \eqref{multiplicativetrace} the following discrete trace inequality. 
\begin{lemma}\label{TraceDG}
There exists a constant $C_{\textup{tr}}>0$ independent of $h$ and $\kappa$ such that 	\begin{equation}\label{discTrace}
  \norm*{\gamma^{\frac{1}{2}}_\cF h^{\frac{1}{2}}_{\cF}\mean{\kappa \bv}}_{0,\cF^*_h}\leq C_{\textup{tr}} \norm*{\kappa^{\frac{1}{2}} \bv}_{0,\Omega}\quad \forall  v\in \cP_k(\cT_h, \R^d). 
 \end{equation}
\end{lemma}
\begin{proof}
	As a consequence of \eqref{multiplicativetrace}, there exists $C_{\textup{tr}}>0$ independent of $h$ such that (see for example \cite{DiPietroErn})
	\begin{equation}\label{Tr0}
		h_K^{\frac{1}{2}}\norm{v}_{0,\partial K} \leq C_{\textup{tr}}  \norm{v}_{0,K}\quad  \forall v \in \cP_k(K). 
	\end{equation}
	By definition of $\gamma_\cF$, for any $\bv \in \cP_k(\cT_h, \R^d)$,
	\begin{align*}
		\norm*{\gamma^{\frac{1}{2}}_\cF h^{\frac{1}{2}}_{\cF}\mean{\kappa \bv}}^2_{0,\cF^*_h} = \sum_{F\in \cF^*_h} h_F \norm*{ \gamma^\frac{1}{2}_{\cF}\mean{\kappa \bv}_F }^2_{0,F} \leq \sum_{F\in \cF^*_h} h_F \norm*{ \mean{\kappa^{\frac{1}{2}} \bv}_F }^2_{0,F}
		\lesssim \sum_{K\in \cT_h} h_K \norm*{ \kappa_K^{\frac{1}{2}} \bv_K }^2_{0,\partial K},
	\end{align*}
and the result follows from \eqref{Tr0}. 
\end{proof}
The Scott--Zhang like quasi-interpolation operator $\pi_h:\, L^2(\Omega) \to \cP_{k}(\cT_h)\cap H^1(\Omega)$, obtained in \cite{ern} by applying an $L^2$-orthogonal projection onto $\cP_{k}(\cT_h)$ followed by an averaging procedure with range in the space of continuous and piecewise $\cP_{k}$ functions,   will be especially useful in the forthcoming analysis. We recall in the next lemma the local approximation properties given \cite[Theorem 5.2]{ern}. Let us first introduce some notations. For any $K\in \cT_h$, we introduce the subset of $\cT_h$ defined by  $\cT_h^K := \set*{K'\in\cT_h:\, K\cap K' \neq \emptyset}$ and let $D_K = \text{interior}\left(\cup_{K'\in \cT_h^K} K' \right)$.   

\begin{lemma}\label{scott-zhang}
	The quasi-interpolation operator $\pi_h$ is invariant in the space $\cP_{k}(\cT_h)\cap H^1(\Omega)$ and there exists a constant $C>0$ independent of $h$ such that 
	\begin{equation}\label{scott0}
		|v - \pi_h v |_{m,K} \leq C h_K^{r - m} |v|_{r,D_K},
	\end{equation}
for all real numbers $0\leq r \leq k+1$, all natural numbers $0\leq m \leq [r]$, all $v\in H^r(D_K)$ and all $K\in \cT_h$. Here $[r]$ stands for the the largest integer less than or equal to $r$. 
\end{lemma}
We point out that, as a consequence of \eqref{scott0} and the triangle inequality, it holds
\begin{equation}\label{scott00}
		|\pi_h v|_{m,K} \lesssim |v|_{m,D_K},
	\end{equation}
	for all natural number $0\leq m \leq k+1$, , all $v\in H^m(D_K)$ and all $K\in \cT_h$. 	 Moreover, it is straightforward to deduce from \eqref{scott0} and the multiplicative trace inequality \eqref{multiplicativetrace} that 
	\begin{equation}\label{scott0b}
		h_K^{\frac{1}{2}}\norm{v - \pi_h v}_{0,\partial K} + h_K^{3/2}\norm{\nabla ( v - \pi_h v )}_{0,\partial K} \lesssim h_K^r |v|_{r,D_K},
	\end{equation}
	for all $2 \leq r \leq k+1$ ($k\geq 1$), all $v\in H^r(D_K)$ and all $K\in \cT_h$. 

We can deduce from \eqref{scott00} a global stability property for $\pi_h$ on $H^m(\cup_j\Omega_j)$,  $0\leq m \leq k+1$, by taking advantage of the fact that the cardinal $\#(\cT_h^K)$ of $\cT_h^K$ is uniformly bounded for all $K\in \cT_h$ and all $h$, as a consequence of the shape-regularity of the mesh sequence $\{\cT_h\}$. Indeed, given $v\in H^m(\cup_j\Omega_j)$, we let $ \cT_h^K(\Omega_j) := \set*{K'\in\cT_h(\Omega_j):\, K\cap K' \neq \emptyset}$ be the subset of elements in $\cT_h^K$ that are contained in $\bar\Omega_j$ and denote $D_K^j := \text{interior}\left(\cup_{K'\in \cT_h^K(\Omega_j)} K' \right)$. It follows from \eqref{scott00} that 
\[
		\norm{\pi_h v}_{m,K} \lesssim \norm{v}_{m,D_K^j},\quad  \forall K\in \cT_h^K(\Omega_j).
\]
Summing over $K\in \cT_h^K(\Omega_j)$  and using that $\#(\cT_h^{K}(\Omega_j)) \leq \#(\cT_h^{K}) \leq c$ for all $1\leq j \leq J$ and all $h$, we deduce that 
\begin{equation}\label{scott0a}
		\norm{\pi_h v}_{m,\Omega_j} \lesssim  \norm{v}_{m,\Omega_j},\quad  \forall j=1,\ldots, J.
\end{equation}
 
 In what follows, we use the same notation for the tensorial version $\pi_h:\, L^{2}(\Omega,\bbS)\to \cP_{k}(\cT_h,\bbS)\cap H^1(\Omega,\bbS)$ of the quasi-interpolation operator, which is obtained by applying the scalar operator componentwise. It is important to notice that such an operator preserves the  symmetry of tensors. As a consequence of \eqref{scott0} and \eqref{scott0b} we have the following result.

\begin{lemma}\label{maintool}
There exists a constant $C>0$ independent of $h$ and $\kappa$ such that 
\begin{equation}\label{tool}
	\vertiii{\btau - \pi_h \btau}_* \leq 
	C  h^{\min\{r,k\}} {\left( \sum_{j=1}^J \kappa_j\norm*{\btau}^2_{r+1,\Omega_j}\right)^2},
\end{equation}
	for all $\btau \in H(\bdiv, \Omega, \bbS)\cap H^{r+1}(\cup_j\Omega_j,\bbM)$, $r\geq 1$,  {provided $h\leq \sqrt \kappa_+$}.
\end{lemma}
\begin{proof}
Given $K\in \cT_h(\Omega_j)$, $1\leq j \leq J$, we obtain from  \eqref{scott0} that  
\begin{equation}\label{esti1}
\norm*{\btau - \pi_h \btau}^2_{0,K}\lesssim h_K^{2\min\{r+1,k+1\}}  \norm*{\btau}^2_{r+1,D^j_K},
\end{equation}
 as well as  
 \begin{align}\label{esti2}
 \begin{split}
 \norm*{\kappa^{\frac{1}{2}}\bdiv\big( \btau - \pi_h \btau \big)}^2_{0,K} &\lesssim 
\norm*{\kappa^{\frac{1}{2}}\bnabla\big( \btau - \pi_h \btau \big)}^2_{0,K}
\lesssim \kappa_j h_K^{2\min\{r,k\}}  \norm*{\btau}^2_{r+1,D^j_K}.
 \end{split}
\end{align} 
For the last term in the left-hand side of \eqref{tool}, we notice that 
\begin{align*}
  \norm*{\gamma_\cF^{\frac{1}{2}} h_\cF^{\frac{1}{2}} \mean{\kappa \bdiv\big(\btau - \pi_h \btau \big)}}^2_{0,\cF^*_h}   
  \lesssim    
  \sum_{j=1}^J \sum_{K\in \cT_h(\Omega_j)}  h_K\norm*{ \kappa^{\frac{1}{2}} \bnabla \Big( \btau -    \pi_h\btau\Big) }^2_{0,\partial K}.
\end{align*}
and \eqref{scott0b} yields  
 \begin{equation}\label{esti3}
 h_K\norm*{ \kappa^{\frac{1}{2}}\bnabla \Big( \btau -    \pi_h\btau \Big) }^2_{0,\partial K}    
 \lesssim \kappa_j h_K^{2\min\{r,k\}} \norm*{\btau}^2_{r+1,D_K^j}, \quad \forall K\in \cT_h(\Omega_j).
\end{equation}
Summing \eqref{esti1}, \eqref{esti2} and \eqref{esti3} over $K\in \cT_h(\Omega_j)$  and then over $j=1,\ldots, J$ and invoking the shape-regularity of the mesh sequence give the result. 
\end{proof}
  
Given $s>1/2$ and $m\geq 1$, the tensorial version of the canonical interpolation operator $\Pi^{\texttt{BDM}}_h: H^s(\cup_j\Omega_j, \R^d) \cap H(\sdiv,\Omega) \to \cP_m(\cT_h,\R^d)\cap H(\sdiv,\Omega)$ associated with the Brezzi--Douglas--Marini (BDM) mixed finite element satisfies the following classical error estimate, see \cite[Proposition 2.5.4]{BoffiBrezziFortinBook},
\begin{equation}\label{asympS}
 \norm*{\bv - \Pi^{\texttt{BDM}}_h \bv}_{0,\Omega} \leq C h^{\min\{s, m+1\}}\left( \sum_{j=1}^J \norm*{\bv}^2_{s,\Omega_j}\right)^{1/2} \quad \forall \bv \in H^s(\cup_j\Omega_j, \R^d) \cap H(\sdiv,\Omega), \quad s>1/2.
\end{equation}
Moreover, we have the well-known commutativity property, 
\begin{equation*}
	\sdiv \Pi^{\texttt{BDM}}_h \bv = Q^{m-1}_h \sdiv \bv,\quad \forall \bv \in H^s(\cup_j\Omega_j, \R^d) \cap H(\sdiv,\Omega), \quad s>1/2, 
\end{equation*}
where $Q^{m-1}_h$ stands for the $L^2(\Omega)$-orthogonal projection onto $\cP_{m-1}(\cT_h)$.
Applying $\Pi^{\texttt{BDM}}_h$ row-wise to matrices we obtain $\bPi^{\texttt{BDM}}_h:\, H^s(\cup_j\Omega_j, \bbM) \cap H(\sdiv,\Omega,\bbM) \to \cP_m(\cT_h,\bbM) \cap H(\sdiv,\Omega,\bbM)$. Obviously, this tensorial version of the Brezzi-Douglas-Marini interpolation operator also satisfies 
\begin{equation}\label{asympV}
 \norm*{\btau - \bPi^{\texttt{BDM}}_h \btau}_{0,\Omega} \leq C h^{\min\{s, m+1\}} \left( \sum_{j=1}^J \norm*{\btau}^2_{s,\Omega_j}\right)^{1/2} \quad \forall \btau \in H^s(\cup_j\Omega_j, \bbM) \cap H(\bdiv,\Omega,\bbM), \quad s>1/2,
\end{equation}
and 
\begin{equation}\label{commutingV}
	\bdiv \bPi^{\texttt{BDM}}_h \btau = \boldsymbol{Q}^{m-1}_h \bdiv \btau,\quad \forall \btau \in H^s(\cup_j\Omega_j, \bbM)\cap H(\bdiv, \Omega, \bbM), \quad s>1/2. 
\end{equation}
where $\boldsymbol{Q}^{m-1}_h$ is the $L^2(\Omega, \R^d)$-orthogonal projection onto $\cP_{m-1}(\cT_h, \R^d)$.

\section{The mixed-DG method and its convergence analysis}\label{sec:dg}

We assume that $\bF\in H^1(\cup_j \Omega_j, \R^d)$ and consider the following DG discretization of \eqref{mvf}: Find $\bsig_h\in \cP_{k}(\cT_h,\bbS)$ such that 
\begin{align}\label{MDG}
\begin{split}
		a(\bsig_h, \btau) &+ \inner*{\kappa \bdiv_h \bsig_h, \bdiv_h\btau }
	- \inner*{\mean{\kappa \bdiv_h \bsig_h}, \jump{\btau} }_{\cF^*_h}
	- \inner*{\mean{\kappa \bdiv_h \btau}, \jump{\bsig_h} }_{\cF^*_h}
	\\[1ex]
	&+ \mathtt{a} \inner*{ \gamma_\cF^{-1}h_\cF^{-1}\jump{\bsig_h}, \jump{\btau}}_{\cF^*_h}
	= - \inner*{\kappa \bF , \bdiv_h\btau } + \inner*{\mean{\kappa \bF}, \jump{\btau} }_{\cF^*_h},
	\quad \forall \btau \in \cP_{k}(\cT_h,\bbS),
\end{split}
\end{align}
where $\mathtt{a}>0$ is a large enough given parameter. In the case $\Gamma_D = \Gamma$, we notice that taking $\btau= \bI$ in 
\eqref{MDG} implies that $\inner*{\tr \bsig_h,1}=0$.

\begin{remark}
  Should the boundary conditions be modified to be non-homogeneous
  \begin{equation}\label{bc:non-homogeneous}
    \frac{\kappa}{\mu}(\bF + \bdiv\bsig) = \bG_D \quad \text{on}\quad \Gamma_D, \qquad \text{and} \qquad
    \bsig\bn = \bG_N \quad \text{on} \quad \Gamma_N,\end{equation}
  for sufficiently regular Dirichlet velocity $\bG_D$ and sufficiently regular normal Cauchy stress $\bG_N$, then  \eqref{MDG} needs to be rewritten as follows: Find $\bsig_h\in \cP_{k}(\cT_h,\bbS)$ such that 
\begin{align}\label{MDG-nonh}
	a(\bsig_h, \btau) &+ \inner*{\kappa \bdiv_h \bsig_h, \bdiv_h\btau }
	- \inner*{\mean{\kappa \bdiv_h \bsig_h}, \jump{\btau} }_{\cF^*_h}
	- \inner*{\mean{\kappa \bdiv_h \btau}, \jump{\bsig_h} }_{\cF^*_h}
\nonumber	\\[1ex]
	&+ \mathtt{a} \inner*{ \gamma_\cF^{-1}h_\cF^{-1}\jump{\bsig_h}, \jump{\btau}}_{\cF^*_h}
	= \langle\mu\,\bG_D,\btau\bn\rangle_{\Gamma_D} - \inner*{\kappa \bF , \bdiv_h\btau } + \inner*{\mean{\kappa \bF}, \jump{\btau} }_{\cF^*_h}\\[1ex]
        &\qquad \qquad \qquad \qquad \qquad \qquad - \langle\kappa\,\bG_N, \bdiv_h\btau\rangle_{\Gamma_N} 
        +\langle \mathtt{a} \gamma_\cF^{-1}h_\cF^{-1}\bG_N, \btau\bn\rangle_{\Gamma_N}        
	\quad \forall \btau \in \cP_{k}(\cT_h,\bbS).\nonumber
\end{align}
  \end{remark}

\begin{prop}\label{SPD}
The linear systems of equations corresponding to \eqref{MDG} are symmetric and positive definite, provided {$\mathtt{a} \geq 2 C_{\textup{tr}}^2 + \frac{1}{2}$}. 
\end{prop}
\begin{proof}
{We first point out that the mapping $\btau_h \mapsto \vertiii{\btau_h}$ defined in \eqref{eq:triple-norm} is actually a norm on  $\cP_k(\cT_h,\bbS)$. Indeed, if $\btau_h \in \cP_k(\cT_h,\bbS)$ satisfies $\vertiii{\btau_h} = 0$, then $\btau_h$ is H(div)-conforming since the jumps of its normal components vanish across all the internal faces $F\in \cF_h^0$. Moreover, it holds $\btau_h\bn = \mathbf 0$ on $\Gamma_N$. Hence,  $\btau_h = \mathbf 0$ as a consequence of \eqref{ellipM}.} 

By virtue of the Cauchy-Schwarz inequality, Young's inequality together with the discrete trace inequality \eqref{discTrace} it holds,
\begin{align}\label{es0}
\begin{split}
	2 \inner*{ \mean{\kappa \bdiv_h \btau_h}, \jump{\btau_h} }_{\cF_h^*} &\leq  2\norm*{\gamma_\cF^{\frac{1}{2}}h_\cF^{\frac{1}{2}} \mean{\kappa \bdiv_h \btau_h} }_{0,\cF^*_h} \norm*{\gamma_\cF^{-\frac{1}{2}} h_\cF^{-\frac{1}{2}} \jump{\btau_h} }_{0,\cF^*_h}
	\\[1ex]
&\leq 2 C_{\textup{tr}}\norm*{\kappa^{\frac{1}{2}} \bdiv \btau_h}_{0,\Omega} \norm*{\gamma_\cF^{-\frac{1}{2}} h_\cF^{-\frac{1}{2}} \jump{\btau_h}}_{0,\cF^*_h}
\\[1ex]
&\leq \tfrac{1}{2} \norm*{\kappa^{\frac{1}{2}} \bdiv \btau_h}^2_{0,\Omega} + 2 C_{\textup{tr}}^2 \norm*{\gamma_\cF^{-\frac{1}{2}} h_\cF^{-\frac{1}{2}} \jump{\btau_h}}^2_{0,\cF^*_h},\quad \forall \btau_h \in \cP_k(\cT_h,\bbS). 
\end{split}
\end{align} 
It follows from \eqref{es0} that 
	 \begin{align}\label{ste}
\begin{split}
		a(\btau_h, \btau_h) + \norm*{\kappa^{\frac{1}{2}} \bdiv_h \btau_h}^2_{0,\Omega} + \mathtt{a} \norm*{ \gamma_\cF^{-\frac{1}{2}} h_\cF^{-\frac{1}{2}}\jump{\btau_h}}^2_{0,\cF^*_h}
	 - 2 \inner*{\mean{\kappa \bdiv_h \btau_h}, \jump{\btau_h} }_{\cF^*_h}& \geq 
	 \\[1ex]
	 a(\btau_h, \btau_h) + \tfrac{1}{2} \norm*{\kappa^{\frac{1}{2}} \bdiv \btau_h}^2_{0,\Omega} + (\mathtt{a} - 2 C_{\textup{tr}}^2 ) \norm*{ \gamma_\cF^{-\frac{1}{2}} h_\cF^{-\frac{1}{2}}\jump{\btau_h}}^2_{0,\cF^*_h} &\geq \tfrac{1}{2} \vertiii{\btau_h}^2>0, 
\end{split}
\end{align}
for all $\mathbf 0 \neq \btau_h\in \cP_k(\cT_h,\bbS)$, if  $\mathtt{a}\geq 2 C_{\textup{tr}}^2 + \frac{1}{2}$, and the result follows. 
\end{proof}

\begin{prop}
	The solution $\bsig$ of \eqref{mvf} satisfies the following consistency property
	\begin{align}\label{consistency}
\begin{split}
		a(\bsig, \btau) &+ \inner*{\kappa \bdiv \bsig, \bdiv_h\btau }_\Omega
	- \inner*{\mean{\kappa \bdiv \bsig}, \jump{\btau} }_{\cF^*_h}
	\\[1ex]
	&= - \inner*{\kappa \bF , \bdiv_h\btau } + \inner*{\mean{\kappa \bF}, \jump{\btau} }_{\cF^*_h},\quad \forall \btau \in \cP_{k}(\cT_h,\bbS).
\end{split}
\end{align}
\end{prop}
\begin{proof}
	 Taking into account that $\inner*{\tr\bsig, 1} = 0$ in the case $\Gamma_D = \Gamma$, and testing \eqref{mvf} with a tensor $\btau:\Omega \to \bbS$ whose entries are  indefinitely differentiable and compactly supported  in $\Omega$, we deduce that $\bu := \frac{\kappa}{{\mu}} (\bdiv \bsig + \bF)$ satisfies $\beps(\bu) = \frac{1}{2{\mu}}\bsig^\tD \in L^2(\Omega,\bbS)$. Moreover, applying a Green's formula in \eqref{mvf}  yields $\bu = \mathbf 0$ on $\Gamma_D$. Hence, by virtue of Korn's inequality, $\bu\in H^1_D(\Omega,\R^d):= \set*{\bv \in  H^1(\Omega,\R^d):\ \bv|_{\Gamma_D} = \mathbf 0}$, which ensures that 
	\[
  \bdiv \bsig = -\bF + {\frac{\mu}{\kappa}}\bu \in H^1(\cup_j \Omega_j, \R^d).
\]
Furthermore, an integration by parts on each element $K\in \cT_h$ gives  
\begin{align*}
	- \tfrac{1}{2} \inner*{\bsig^\tD, \btau} =   - \inner*{{\mu}\beps(\bu), \btau} 
    =  \inner*{{\mu}\bu, \bdiv_h\btau } - \inner*{{\mu}\mean{\bu}, \jump{\btau} }_{\cF^*_h} \quad   \forall \btau\in \cP_{k}(\cT_h,\bbS),
\end{align*}
and the result follows by substituting $\bu := {\frac{\kappa}{\mu}} (\bdiv \bsig + \bF)$ in the last expression. 
\end{proof}

\begin{lemma}
There exists $\mathtt{a}_0>0$ such that 
\begin{equation}\label{stab0}
	 \vertiii{\bsig - \bsig_h} \leq C \vertiii{\bsig - \pi_h \bsig}_*,
\end{equation}
holds true for all $\mathtt{a} \geq  \mathtt{a}_0$, with $C$ and $\mathtt{a}_0$  independent of $h$, $\kappa$, {and $\mu$}. 
\end{lemma}
\begin{proof}
We introduce $\be_h = \pi_h\bsig - \bsig_h \in \cP_k(\cT_h, \bbS)$ and 
notice that, thanks to \eqref{consistency}, 
\begin{align}\label{step1}
\begin{split}
		a(\be_h, \btau) &+ \inner*{\kappa \bdiv_h \be_h, \bdiv_h\btau }
	- \inner*{\mean{\kappa \bdiv_h \be_h}, \jump{\btau} }_{\cF^*_h}
	\\[1ex]
	&- \inner*{\mean{\kappa \bdiv_h \btau}, \jump{\be_h} }_{\cF^*_h}
	+ \mathtt{a} \inner*{ \gamma_\cF^{-1} h_\cF^{-1}\jump{\be_h}, \jump{\btau_h}}_{\cF^*_h}
	= G(\btau) \quad  \forall \btau\in \cP_{k}(\cT_h,\bbS),
\end{split}
\end{align}
where
\[
  G(\btau) :=- a(\bsig - \pi_h \bsig, \btau) - \inner*{\kappa\bdiv (\bsig - \pi_h \bsig), \bdiv_h \btau }  + \inner*{\mean{\kappa \bdiv (\bsig - \pi_h \bsig)}, \jump{\btau} }_{\cF^*_h}.
\]
Taking $\btau = \be_h$ in \eqref{step1} yields
 \begin{align}\label{step2}
\begin{split}
		a(\be_h, \be_h) &+ \norm*{\kappa^{\frac{1}{2}} \bdiv_h \be_h}^2_{0,\Omega} + \mathtt{a} \norm*{ \gamma_\cF^{-\frac{1}{2}} h_\cF^{-\frac{1}{2}}\jump{\be_h}}^2_{0,\cF^*_h}
	 = 2 \inner*{\mean{\kappa \bdiv_h \be_h}, \jump{\be_h} }_{\cF^*_h} + G(\be_h).
\end{split}
\end{align}
Let us bound now each of terms of the right-hand side of \eqref{step2} by means of the Cauchy--Schwarz and Young inequalities. For the first term, proceeding as in  \eqref{es0} gives 
\begin{align}\label{es1}
\begin{split}
	2 \inner*{ \mean{\kappa \bdiv_h \be_h}, \jump{\be_h} }_{\cF_h^*} \leq \tfrac{1}{4} \norm*{\kappa^{\frac{1}{2}} \bdiv \be_h}^2_{0,\Omega} + 4 C_{\textup{tr}}^2 \norm*{\gamma_\cF^{-\frac{1}{2}} h_\cF^{-\frac{1}{2}} \jump{\be_h}}^2_{0,\cF^*_h}. 
\end{split}
\end{align} 
Next, we estimate $G(\be_h)$ in two steps{: from the one hand,}  
\begin{align}\label{es2}
\begin{split}
&	- a(\bsig - \pi_h \bsig, \be_h) - \inner*{\kappa\bdiv (\bsig - \pi_h \bsig), \bdiv_h \be_h } 
	\\[1ex]
	&\qquad  \leq  \tfrac{1}{2} a(\be_h, \be_h) + \tfrac{1}{4} \norm*{\kappa^{\frac{1}{2}}\bdiv_h \be_h}^2_{0,\Omega}  + \tfrac{1}{2} a(\bsig - \pi_h \bsig, \bsig - \pi_h \bsig )+ \norm*{\kappa^{\frac{1}{2}}\bdiv (\bsig - \pi_h \bsig)}^2_{0,\Omega},
\end{split}
\end{align}
{and from the other hand, }
\begin{align}\label{es3}
\begin{split}
	\inner*{\mean{\kappa \bdiv (\bsig - \pi_h \bsig)}, \jump{\be_h} }_{\cF^*_h} &\leq \norm*{\gamma_\cF^{-\frac{1}{2}} h_\cF^{-\frac{1}{2}} \jump{\be_h}}_{0,\cF^*_h}
	\norm*{\gamma_\cF^{\frac{1}{2}} h_\cF^{\frac{1}{2}} \mean{\kappa \bdiv (\bsig - \pi_h \bsig)}}_{0,\cF^*_h}
	\\[1ex]
	& \leq \tfrac{1}{2}\norm*{\gamma_\cF^{-\frac{1}{2}} h_\cF^{-\frac{1}{2}} \jump{\be_h}}^2_{0,\cF^*_h} + \tfrac{1}{2}\norm*{\gamma_\cF^{\frac{1}{2}} h_\cF^{\frac{1}{2}} \mean{\kappa \bdiv (\bsig - \pi_h \bsig)}}^2_{0,\cF^*_h}.
\end{split}
\end{align}
Substituting \eqref{es1}-\eqref{es3} in \eqref{step2} and rearranging terms we deduce that \eqref{stab0} is satisfied if  $\mathtt{a} \geq \mathtt{a}_0:= 1 + 4 C_{\textup{tr}}^2$. 
\end{proof}

\begin{theorem}\label{conv}
Let $\bsig$ and $\bsig_h$ be the solutions of problems \eqref{mvf} and  \eqref{MDG}, respectively. If $\bsig\in  H^{r+1}(\cup_j\Omega_j,\bbM)$, with $r\geq 1$, then 
\begin{align}\label{asympSD}
\begin{split}
\vertiii{\bsig  - \bsig_h}    \leq C  \, {\sqrt \kappa_+} h^{\min\{r,k\}}
{\left( \sum_{j=1}^J   \norm*{ \bsig}^2_{r+1, \Omega_j}\right)^{1/2}} ,
\end{split}
\end{align}
for all $\mathtt{a} \geq \mathtt{a}_0$, with $C>0$ independent of $h$, $\kappa$, {and $\mu$}.
\end{theorem}
\begin{proof}
	The result is a direct combination of \eqref{stab0} and the interpolation error estimates provided by Lemma~\ref{maintool}. 
\end{proof}

\section{Discrete post-processing of velocity and pressure}\label{sec:postpr}
We recall that the velocity field and the pressure can be recovered from the stress tensor  at the continuous level  by
\[
  \bu := {\frac{\kappa}{\mu}} (\bdiv \bsig + \bF )\in H_0^1(\Omega, \R^d) \quad \text{and} \quad p := -\tfrac{1}{d} \tr \bsig\in L^2(\Omega).
\]
These same expressions permit us to reconstruct, with local and independent calculations on each element, a discrete velocity $\bu_h$ and a discrete pressure $p_h$ that converge to their continuous counterpart at an optimal rate of convergence, as shown in the following results. 
\begin{corollary}\label{postV}
{Let $\bsig$ and $\bsig_h$ be the solutions of problems \eqref{mvf} and  \eqref{MDG}, respectively.} We introduce $\bu_h := {\frac{\kappa}{\mu}} (\bdiv_h \bsig_h + Q^{k-1}_h\bF )$. If $\bsig\in  H^{r+1}(\cup_j\Omega_j,\bbM)$ and $\bF\in  H^{r}(\cup_j\Omega_j,\R^d)$ with $r\geq 1$,  there exists a constant $C>0$ independent of $h$, $\kappa$, {and $\mu$} such that, for all $\mathtt{a} \geq  \mathtt{a}_0$,  
\[
  \norm{\bu - \bu_h}_{0,\Omega} \leq C {\frac{\kappa_+}{\mu}} \,  h^{\min\{r,k\}} \left( 
\sum_{j=1}^J    \norm*{ \bsig}^2_{r+1, \Omega_j} + \norm*{ \bF}^2_{r, \Omega_j}  \right)^{1/2}.
\]
\end{corollary}
\begin{proof}
	The result follows immediately from \eqref{asympSD} and from the fact that
	\[
  \norm*{\kappa(\bF - Q^{k-1}_h\bF)}_{0,\Omega} \lesssim   \,  h^{\min\{r,k\}} 
\left( \sum_{j=1}^J   \kappa_j^2 \norm*{ \bF}^2_{r, \Omega_j} \right)^{1/2}.
\]

\end{proof}

\begin{corollary}\label{postP}
{Let $\bsig$ and $\bsig_h$ be the solutions of problems \eqref{mvf} and  \eqref{MDG}, respectively.} We consider $p_h := -\tfrac{1}{d} \tr \bsig_h$. If $\bsig\in  H^{r+1}(\cup_j\Omega_j,\bbM)$  with $r\geq 1$,  there exists a constant $C>0$ independent of $h$  {and $\mu$, but depending on $\kappa_+/\kappa_-$}   such that, for all $\mathtt{a} \geq  \mathtt{a}_0$,  
\[
  \norm{p - p_h}_{0,\Omega} \leq C \,  h^{\min\{r,k\}}
\left( \sum_{j=1}^J   \norm*{ \bsig}^2_{r+1, \Omega_j} \right)^{1/2}.
\]
\end{corollary}
\begin{proof}
We notice that $p - p_h = \tfrac{1}{d} \tr(\bsig_h - \bsig) \in Q$ where
\[
  Q:= \begin{cases}
  	L^2(\Omega) & \text{if $\Gamma_D \neq \Gamma$},
  	\\
  	L^2_*(\Omega) = \set*{\phi \in L^2(\Omega):\ \inner*{ \phi, 1} = 0} & \text{if $\Gamma_D = \Gamma$}.
  \end{cases}
\]
It follows from the well-known inf-sup condition (see for example \cite[Lemma 53.9]{ern1})   
\[
  \sup_{\bv \in H^1_{D}(\Omega,\R^d)} \frac{\inner*{\phi, \sdiv \bv}}{\norm*{\bv}_{1,\Omega}} \geq \delta \norm{\phi}_{0,\Omega},\quad \forall \phi \in Q,
\]
that
\begin{equation}\label{FJS}
	 \delta  \norm{p - p_h}_{0,\Omega} \leq \sup_{\bv \in H^1_D(\Omega,\R^d)} \frac{\inner*{p - p_h, \sdiv \bv}}{\norm*{\bv}_{1,\Omega}}.
\end{equation}
Using an elementwise integration by parts formula gives
\begin{align*}
\inner*{p - p_h, \sdiv \bv} &= -\tfrac{1}{d} \inner*{\tr (\bsig - \bsig_h) \bI, \beps( \bv)}
= \inner*{(\bsig - \bsig_h)^\tD , \beps( \bv)} - \inner*{\bsig - \bsig_h , \beps( \bv)}
\\[1ex]
& 
= \inner*{(\bsig - \bsig_h)^\tD , \beps( \bv)} + \inner*{\bdiv_h (\bsig - \bsig_h) , \bv}
- \inner*{\bv, \jump{\bsig - \bsig_h}}_{\cF^*_h},\quad \forall \bv \in H^1_D(\Omega, \R^d).
\end{align*}
Hence, using the Cauchy--Schwarz inequality 
\begin{align*}
  \inner*{p - p_h, \sdiv \bv} \leq \norm{(\bsig - \bsig_h)^\tD}_{0,\Omega} \norm{\beps( \bv)}_{0,\Omega} &+ \norm{\bdiv_h(\bsig - \bsig_h)}_{0,\Omega} \norm{\bv}_{0,\Omega}, 
  + \norm{h_\cF^{-\frac{1}{2}}\jump{\bsig - \bsig_h}}_{0,\cF^*_h} \norm{h_\cF^{\frac{1}{2}}\bv}_{0,\cF^*_h},
\end{align*}
and the multiplicative trace inequality \eqref{multiplicativetrace} to estimate the term $\norm{h_\cF^{\frac{1}{2}}\bv}_{0,\cF^*_h}$, we deduce that  
\[
  \inner*{p - p_h, \sdiv \bv} \lesssim \Big(\norm{(\bsig - \bsig_h)^\tD}_{0,\Omega} +  \norm{\bdiv_h(\bsig - \bsig_h)}_{0,\Omega} + \norm{h_\cF^{-\frac{1}{2}}\jump{\bsig - \bsig_h}}_{0,\cF^*_h}\Big) \norm{\bv}_{1,\Omega}.
\]
Plugging the forgoing estimate in \eqref{FJS} yields
\[
  \norm{p - p_h}_{0,\Omega} \lesssim  \Big(\norm{(\bsig - \bsig_h)^\tD}_{0,\Omega} +  \norm{\bdiv_h(\bsig - \bsig_h)}_{0,\Omega} + \norm{h_\cF^{-\frac{1}{2}}\jump{\bsig - \bsig_h}}_{0,\cF^*_h}\Big),
\]
and it follows from \eqref{asympSD} that there exists a constant $C>0$ independent of $h$, but depending on $\kappa_+/\kappa_-$ such that
\[
   \norm{p - p_h}_{0,\Omega} \leq  C 
h^{\min\{r,k\}} \left( \sum_{j=1}^J   \norm*{ \bsig}^2_{r+1, \Omega_j} \right)^{1/2}, 
\]
which gives the result.
\end{proof}

We aim now to obtain a velocity field  $\bu_h^*\simeq \bu$ that preserves exactly the incompressibility condition $\sdiv \bu = 0$ in $\Omega$. The computational cost for such an enhancement is more demanding. For $k\geq 2$, it is achieved by solving the following elliptic problem in mixed form: Find  $\bu_h^*\in H(\sdiv,\Omega)\cap \cP_{k-1}(\cT_h, \R^d)$ and $\lambda_h\in \cP_{k-2}(\cT_h)$ such that 
\begin{align}\label{div0}
\begin{split}
	\inner*{\bu_h^*, \bv} + \inner*{\lambda_h, \sdiv \bv} &= \inner*{\bu_h, \bv}\quad \forall \bv \in H(\sdiv,\Omega)\cap \cP_{k-1}(\cT_h, \R^d),
	\\[1ex]
	\inner*{\sdiv \bu_h^*, \eta} &= 0\quad \forall \eta \in \cP_{k-2}(\cT_h). 
	\end{split}
\end{align}
Clearly, $ \sdiv \bu_h^* = 0$ in $\Omega$. We need now to estimate the error $\bu - \bu_h^*$ in the $L^2(\Omega,\R^d)$-norm.  

\begin{lemma}\label{postV2}
{Let $\bsig$ and $\bsig_h$ be the solutions of problems \eqref{mvf} and  \eqref{MDG}, respectively.}
If $\bsig\in  H^{r+1}(\cup_j\Omega_j,\bbM)$ and $\bF\in  H^{r}(\cup_j\Omega_j,\R^d)$ with $r\geq 1$,  there exists a constant $C>0$ independent of $h$,  $\kappa$, {and $\mu$}  such that, for all $\mathtt{a} \geq  \mathtt{a}_0$,  
\[
  \norm{\bu - \bu^*_h}_{0,\Omega} \leq C\, {\max\{1, \frac{\kappa_+}{\mu}\}} \,  h^{\min\{r,k\}}
\left( \sum_{j=1}^J    \norm*{ \bsig}^2_{r+1, \Omega_j} + \norm*{ \bF}^2_{r, \Omega_j} \right)^{1/2},
\]
{where $\bu_h^*$ is obtained by solving the auxiliary problem \eqref{div0}. }
\end{lemma}
\begin{proof}
	We begin by considering the following auxiliary problem: 
Find  $\bw\in H(\sdiv,\Omega)$ and $\lambda\in L^2(\Omega)$ such that 
\begin{align}\label{div0a}
\begin{split}
	\inner*{\bw, \bv} + \inner*{\lambda, \sdiv \bv} &= \inner*{\bu, \bv}\quad \forall \bv \in H(\sdiv,\Omega),
	\\[1ex]
	\inner*{\sdiv \bw, \eta} &= 0\quad \forall \eta \in L^2(\Omega). 
	\end{split}
\end{align}
We denote by $V$ the kernel of the bilinear form $H(\sdiv,\Omega) \times L^2(\Omega) \ni  (\bv, \eta)\mapsto \inner*{\eta, \sdiv \bv}$, that is 
\[
V := \set*{ \bv\in H(\sdiv,\Omega): \  \sdiv \bv =  0  \ \text{in $\Omega$}}.
\]
The fact that $(\bv,\bv) = \norm{\bv}_{0,\Omega}^2 = \norm{\bv}^2_{H(\sdiv, \Omega)}$ for all $\bv \in V$ and the well-known inf-sup condition (cf. \cite{BoffiBrezziFortinBook})
\begin{equation*}\label{discInfSup00}
\sup_{\btau \in H(\sdiv, \Omega)} 
\frac{ \inner*{\sdiv \bv, \eta}}{\norm{\btau}_{H(\sdiv,\Omega)}} \geq \beta 
\norm{\eta}_{0,\Omega} ,\quad \forall \eta \in L^2(\Omega),
\end{equation*}
permit us to apply the Babu\v{s}ka--Brezzi theory to deduce that  problem \eqref{div0a} is well-posed. Moreover, it can easily be seen that its unique solution is given by $\bw = \bu$ and $\lambda = 0$.

We recall that the BDM-mixed finite element pair $\set*{H(\sdiv, \Omega)\cap \cP_{k-1}(\cT_h, \R^d), \cP_{k-2}(\cT_h)}$ satisfies the discrete inf-sup condition 
\begin{equation*}\label{discInfSup}
\sup_{\btau \in H(\sdiv, \Omega)\cap \cP_{k-1}(\cT_h, \R^d)} 
\frac{ \inner*{\sdiv \bv, \eta}}{\norm{\btau}_{H(\sdiv,\Omega)}} \geq \beta' 
\norm{\eta}_{0,\Omega} ,\quad \forall \eta \in \cP_{k-2}(\cT_h), \quad k \geq 2.
\end{equation*}
Moreover, we notice that discrete kernel $V_h$ of the bilinear form $H(\sdiv, \Omega)\cap \cP_{k-1}(\cT_h, \R^d) \times \cP_{k-2}(\cT_h) \ni  (\bv, \eta)\mapsto \inner*{\eta, \sdiv \bv}$ is a subspace of $V$ since
\[
V_h := \set*{ \bv_h\in H(\sdiv, \Omega)\cap \cP_{k-1}(\cT_h, \R^d): \  \sdiv \bv_h =  0  \ \text{in $\Omega$}}.
\]

Hence, the Galerkin method based on the BDM-element and defined by:
Find  $\bw_h\in H(\sdiv, \Omega)\cap \cP_{k-1}(\cT_h, \R^d)$ and $\lambda_h\in \cP_{k-2}(\cT_h)$ such that 
\begin{align}\label{div0b}
\begin{split}
	\inner*{\bw_h, \bv} + \inner*{\lambda_h, \sdiv \bv} &= \inner*{\bu, \bv}\quad \forall \bv \in H(\sdiv, \Omega)\cap \cP_{k-1}(\cT_h, \R^d),
	\\[1ex]
	\inner*{\sdiv \bw_h, \eta} &= 0\quad \forall \eta \in \cP_{k-2}(\cT_h),
	\end{split}
\end{align}
is stable, convergent and we have the C\'ea estimate (recall that $\bw=\bu$ and $\lambda = 0$)
\begin{align*}
	\norm{\bu - \bw_h}_{0,\Omega}  = \norm{\bu - \bw_h}_{H(\sdiv, \Omega)}  \lesssim  \norm{\bu - \Pi_h^{\texttt{BDM}} \bu}_{H(\sdiv, \Omega)} =  \norm{\bu - \Pi_h^{\texttt{BDM}} \bu}_{0, \Omega}.
\end{align*}

We notice now that problem~\eqref{div0} is none other than a discretization of problem~\eqref{div0a} obtained from the Galerkin method \eqref{div0b} after replacing the right-hand side $\inner*{\bu, \bv}$ by $\inner*{\bu_h, \bv}$ for all $\bv\in H(\sdiv, \Omega)\cap \cP_{k-1}(\cT_h, \R^d)$. Hence, by virtue of Strang's Lemma, it immediately holds
\begin{align*}
	\norm{\bu - \bw_h}_{0,\Omega}=\norm{\bu - \bu^*_h}_{H(\sdiv, \Omega)}  &\lesssim 
	\norm{\bu - \Pi^{BDM}_h \bu}_{0, \Omega}  + \norm{\bu - \bu_h}_{0,\Omega}.
\end{align*}
Finally, we deduce from \eqref{asympS} and Corollary~\ref{postV} that  there exists a constant $C>0$ independent of $h$,  $\kappa$, {and $\mu$} such that 
\[
  \norm{\bu - \bu^*_h}_{0,\Omega} \leq C\, {\max\biggl\{1, \frac{\kappa_+}{\mu}\biggr\}} \,  h^{\min\{r,k\}}
\left( \sum_{j=1}^J    \norm*{ \bsig}^2_{r+1, \Omega_j} + \norm*{ \bF}^2_{r, \Omega_j} \right)^{1/2},
\]
and the result follows.
\end{proof}

\section{$L^2$-error estimates for the stress}\label{sec:L2}
The analysis of this section is restricted to the case $\Gamma_D = \Gamma$, so that $\theta = 1$, $X = H(\bdiv,\Omega, \bbS)$ and $\cF_h^* = \cF_h^0$. The deduction of error estimates for the stress in the $L^2$-norm relies on the construction of adequate approximations of the exact solution $\bsig\in H(\bdiv, \Omega, \bbS)$ and of the discrete solution $\bsig_h\in \cP_k(\cT_h, \bbS)$ in the space $H(\bdiv, \Omega, \bbS)\cap \cP_k(\cT_h, \bbS)$. 

One of the tools that are needed to achieve this is the averaging operator $\mathcal{A}_h:\, \cP_k(\cT_h,\bbM) \to \cP_k(\cT_h,\bbM)\cap H(\bdiv,\Omega,\bbM)$ defined on each $K\in \cT_h$ and for any $\btau\in \cP_k(\cT_h,\bbM)$, by the conditions
\begin{subequations}
\begin{eqnarray}
 \int_F (\mathcal{A}_h\btau) \bn_K \cdot \boldsymbol{\phi} & = & \int_F \mean{\btau}_F \bn_K \cdot \boldsymbol{\phi}
\quad \forall \boldsymbol{\phi}\in \cP_k(F,\R^d) \quad\forall F\in \cF(K), \label{GDFa}
 \\
\int_K \mathcal{A}_h\btau : \bnabla \bv &= & \int_K \btau : \bnabla \bv \quad\forall \bv\in \cP_{k-1}(K,\R^d),\label{GDFc}
\\
\int_K \mathcal{A}_h\btau : \boldsymbol{\xi} &= & \int_K \btau : \boldsymbol{\xi} \quad\forall \boldsymbol{\xi}\in \cP_k(K,\bbM),\quad \text{$\bdiv \boldsymbol{\xi} = \boldsymbol{0}$ in $K$},\quad \text{$\boldsymbol{\xi}\bn = \boldsymbol{0}$ on $\partial K$}. \label{GDFd}
\end{eqnarray}\end{subequations}

\begin{lemma}\label{propC}
The projector $\mathcal{A}_h:\, \cP_k(\cT_h,\bbM) \to \cP_k(\cT_h,\bbM)\cap H(\bdiv,\Omega,\bbM)$ is uniquely characterized by the conditions \eqref{GDFa}-\eqref{GDFd} and it satisfies 
\begin{equation}\label{L2Ph}
   \norm*{\btau- \mathcal{A}_h \btau}_{0,\Omega}
 \leq C  h
 \norm*{h_{\cF}^{-\frac{1}{2}} \jump{\btau}}_{0,\cF^0_h}\quad \forall \btau \in \cP_k(\cT_h,\bbM),
\end{equation}
 with $C>0$ independent of $h$.
\end{lemma}
\begin{proof}
See \cite[Proposition 5.2]{MMT}.
\end{proof}

A combination of \eqref{L2Ph} and \eqref{asympSD} shows that, under sufficient regularity assumption on the exact solution $\bsig\in H(\bdiv, \Omega, \bbS)$, $\mathcal{A}_h\bsig_h-\bsig_h$ converges to zero in the $L^2$-norm. However, it is clear that the operator $\mathcal{A}_h$ does not preserve symmetry. To remedy this drawback, we follow \cite{falk, guzman, Qian, wang} and use a symmetrization procedure that requires the stability the Scott--Vogelius element \cite{vogelius} for the Stokes problem. We refer to \cite[Section 55.3]{ern1} for a detailed account on the conditions (on the mesh $\cT_h$ and on $k$) under which this stability property is guaranteed in the 2D and 3D cases. 

\begin{lemma}\label{stokes}
Assume that $\set*{\cP_{k+1}(\cT_h,\R^d)\cap H^1(\Omega,\R^d),  \cP_{k}(\cT_h,\R^d)}$ is a stable Stokes pair on the mesh $\cT_h$. Then, there exists a linear operator
\[
  \mathcal S_h:\, \cP_k(\cT_h,\bbM)\cap H(\bdiv,\Omega,\bbM)\to \cP_{k}(\cT_h,\bbS)\cap H(\bdiv,\Omega,\bbS),
\]
such that, for all $\btau_h \in \cP_k(\cT_h,\bbM)\cap H(\bdiv,\Omega,\bbM)$,
	\begin{enumerate}[label=\roman*)]
		\item $\bdiv ( \btau_h - \mathcal S_h \btau_h) = \mathbf 0$ in $\Omega$, 
		\item  and 
		$
  \norm*{\btau_h - \mathcal S_h \btau_h}_{0,\Omega}\lesssim  \norm*{\btau_h - \btau_h^{\mt}}_{0,\Omega}.
  $  
	\end{enumerate}
\end{lemma}
\begin{proof}
	See for example \cite[Lemma 5.2]{wang}. 
\end{proof}

\begin{remark}
	The reason we are limiting the analysis of this section to the case $\Gamma_D = \Gamma$ is due to the fact that, {to the authors knowledge, the symmetrization procedure provided by operator $\mathcal S_h$ has only been addressed in the literature for homogeneous Dirichlet or Neumann boundary conditions.}
\end{remark}

The next result shows that, under sufficient regularity assumptions on the exact solution $\bsig$ of problem \eqref{mvf},   $\bsig_h - \mathcal S_h(\mathcal{A}_h \bsig_h)$ and $\bsig - \mathcal S_h(\bPi^{\texttt{BDM}}_h \bsig )$ converge to zero at  optimal order in the $L^2$-norm.  

\begin{corollary}\label{inter}
{Let $\bsig$ and $\bsig_h$ be the solutions of problems \eqref{mvf} and  \eqref{MDG}, respectively.}
Assume that the pair $\set*{\cP_{k+1}(\Omega,\R^d)\cap H^1(\Omega,\R^d),  \cP_{k}(\Omega,\R^d)}$ is  Stokes-stable on the mesh $\cT_h$. Then, if $\bsig\in  H^{r+1}(\cup_j\Omega_j,\bbM)$, with $r\geq 1$, there exists $C>0$ independent of $h$ {and $\mu$, but depending on $\kappa_+/\kappa_-$},   such that  
\begin{align}\label{asympSD0}
\begin{split}
\norm*{\bsig - \mathcal S_h(\bPi^{\texttt{BDM}}_h \bsig )}_{0,\Omega} + 
\norm*{\bsig_h - \mathcal S_h(\mathcal{A}_h \bsig_h)}_{0,\Omega}
   \leq C  \,  h^{\min\{r+1,k+1\}}
\left( \sum_{j=1}^J   \norm*{ \bsig}^2_{r+1, \Omega_j}\right)^{1/2},
\end{split}
\end{align}
for all $\mathtt{a} \geq \mathtt{a}_0$.
\end{corollary}
\begin{proof}
We point out that, as a consequence of property ii) in Lemma~\ref{stokes} and the symmetry of $\bsig$, we have that  
\begin{align}\label{hu1}
\begin{split}
	\norm*{\bsig - \mathcal S_h(\bPi^{\texttt{BDM}}_h \bsig )}_{0,\Omega} &\leq \norm*{\bsig - \bPi^{\texttt{BDM}}_h \bsig }_{0,\Omega} + \norm*{\bPi^{\texttt{BDM}}_h \bsig - \mathcal S_h(\bPi^{\texttt{BDM}}_h \bsig )}_{0,\Omega}
	\\[1ex]
	& \lesssim \norm*{\bsig - \bPi^{\texttt{BDM}}_h \bsig }_{0,\Omega} + \norm*{\bsig - \bPi^{\texttt{BDM}}_h \bsig - (\bsig - \bPi^{\texttt{BDM}}_h \bsig)^\mt}_{0,\Omega}\lesssim \norm*{\bsig - \bPi^{\texttt{BDM}}_h \bsig }_{0,\Omega}.
	\end{split}
\end{align}
 Using this time property ii) of Lemma~\ref{stokes} in combination with the symmetry of $\bsig_h$ and \eqref{L2Ph} give 
 \begin{align}\label{hu2}
\begin{split}
	\norm*{\mathcal S_h(\mathcal{A}_h \bsig_h) - \bsig_h}_{0,\Omega} &\leq \norm*{\bsig_h - \mathcal{A}_h \bsig_h }_{0,\Omega} + \norm*{\mathcal{A}_h \bsig_h - \mathcal S_h(\mathcal{A}_h \bsig_h )}_{0,\Omega}
	\\[1ex]
	& \lesssim \norm*{\bsig_h - \mathcal{A}_h \bsig_h }_{0,\Omega} \lesssim h \norm*{h_\cF^{\frac{1}{2}}\jump{\bsig_h}}_{0,\cF^0_h} = h \norm*{h_\cF^{\frac{1}{2}}\jump{\bsig - \bsig_h}}_{0,\cF^0_h}.
	\end{split}
\end{align}
The result follows now by using \eqref{asympV} in  \eqref{hu1} and \eqref{asympSD} in \eqref{hu2}. 
\end{proof}

\begin{theorem}\label{L2}
{Let $\bsig$ and $\bsig_h$ be the solutions of problems \eqref{mvf} and  \eqref{MDG}, respectively.}
Assume that the pair  $\set*{\cP_{k+1}(\Omega,\R^d)\cap H^1(\Omega,\R^d),  \cP_{k}(\Omega,\R^d)}$ is Stokes-stable on the mesh $\cT_h$. Then, if $\bsig\in  H^{r+1}(\cup_j\Omega_j,\bbM)$, with $r\geq 1$, there exists $C>0$ independent of $h$ {and $\mu$, but depending on $\kappa_+/\kappa_-$},   such that 
\begin{align*}
\norm*{\bsig^\tD  - \bsig_h^\tD}_{0,\Omega}    \leq C  \,  h^{\min\{r+1,k+1\}}
\left( \sum_{j=1}^J   \norm*{ \bsig}^2_{r+1, \Omega_j}\right)^{1/2},
\end{align*}
for all $\mathtt{a} \geq \mathtt{a}_0$.
\end{theorem}
\begin{proof}
It follows from the consistency property  \eqref{consistency} that 
\begin{align}\label{s-sh}
	a(\bsig - \bsig_h, \btau_h) + \inner*{\kappa (\bdiv \bsig - \bdiv_h \bsig_h), \bdiv \btau_h} + \inner*{\mean{\kappa \bdiv \btau }, \jump{\bsig_h}}_{\cF^0_h} = 0,
\end{align}
for all $\btau_h \in \cP_{k}(\cT_h,\bbS)\cap H(\bdiv,\Omega, \bbS)$. Using an integration by parts in each $K\in \cT_h$  gives 
\begin{align*}
	\inner*{\kappa \bdiv_h \bsig_h, \bdiv \btau_h} &= \sum_{K\in \cT_h}  \int_K \kappa\bdiv \bsig_h\cdot \bdiv \btau_h
	\\[1ex]
	&= - \sum_{K\in \cT_h}  \int_K \kappa\bsig_h : \nabla (\bdiv \btau_h) + 
	\sum_{K\in \cT_h}  \int_{\partial K} \kappa\bsig_h\bn_K \cdot \bdiv \btau_h,
\end{align*}
and taking advantage of \eqref{GDFc} yields  
\begin{align}\label{equiv}
\begin{split}
	\inner*{\kappa \bdiv_h \bsig_h, \bdiv \btau_h} &= - \sum_{K\in \cT_h}  \int_K \kappa\mathcal{A}_h \bsig_h : \nabla (\bdiv \btau_h) + 
	\sum_{K\in \cT_h}  \int_{\partial K} \kappa\bsig_h\bn_K \cdot \bdiv \btau_h
	\\[1ex]
	&= \int_K \kappa \bdiv \mathcal{A}_h \bsig_h \cdot  \bdiv \btau_h  + 
	\sum_{K\in \cT_h}  \int_{\partial K} (\bsig_h - \mathcal{A}_h \bsig_h)\bn_K \cdot \kappa \bdiv \btau_h,
	\end{split}
\end{align}
for all $\btau_h \in \cP_{k}(\cT_h,\bbS)\cap H(\bdiv,\Omega, \bbS)$. 
Next, by construction of $\mathcal{A}_h$ it holds  
\[
 \int_F(\bsig_h- \mathcal{A}_h \bsig_h)\bn_K\cdot 
 \boldsymbol{\phi} = \begin{cases}
 \frac{1}{2}\displaystyle\int_F \jump{\bsig_h}_F\cdot \boldsymbol{\phi}\quad \forall \boldsymbol{\phi}\in \cP_k(F,\R^d) & \text{if $F\in \cF^0_h$},\\[1ex]
 0\quad \forall \boldsymbol{\phi}\in \cP_k(F,\R^d) & \text{if $F\in \cF^\partial_h$},
 \end{cases}
\] 
which means that \eqref{equiv} can be equivalently written  
\[
  \inner*{\kappa \bdiv_h \bsig_h, \bdiv \btau_h} = \inner*{ \kappa \bdiv \mathcal{A}_h \bsig_h, \bdiv \btau_h} + \inner*{\mean{\kappa \bdiv \btau_h }, \jump{\bsig_h}}_{\cF^0_h}.
\]
Substituting back the last identity in \eqref{s-sh} yields
\begin{align}\label{s-sh2}
	a(\bsig - \bsig_h, \btau_h) + \inner*{\kappa \bdiv (\bsig - \mathcal{A}_h\bsig_h), \bdiv \btau_h}  = 0 
	\qquad \forall \btau_h \in \cP_{k}(\cT_h,\bbS)\cap H(\bdiv,\Omega, \bbS).
\end{align}
On the other hand, we notice that by virtue of \eqref{commutingV} and property i) of Theorem~\ref{stokes}, 
\begin{align*}
	\inner*{\kappa \bdiv (\bsig -  \mathcal{A}_h\bsig_h), \bdiv \btau_h} &= \inner*{\kappa \bdiv( \bPi^{\texttt{BDM}}_h \bsig - \mathcal{A}_h\bsig_h), \bdiv \btau_h}
	\\[1ex]
	& = \inner*{\kappa \bdiv( \mathcal S_h(\bPi^{\texttt{BDM}}_h \bsig - \mathcal{A}_h\bsig_h)), \bdiv \btau_h},
\end{align*}
and hence, taking $\btau_h =  \mathcal S_h(\bPi^{\texttt{BDM}}_h \bsig - \mathcal{A}_h\bsig_h) \in \cP_{k}(\cT_h,\bbS)\cap H(\bdiv,\Omega, \bbS)$ in \eqref{s-sh2} gives  
\begin{align}\label{s-sh3}
	a(\bsig - \bsig_h, \mathcal S_h(\bPi^{\texttt{BDM}}_h \bsig - \mathcal{A}_h\bsig_h)) + \norm*{\kappa^{\frac{1}{2}} \bdiv (\mathcal S_h(\bPi^{\texttt{BDM}}_h \bsig - \mathcal{A}_h\bsig_h))}^2_{0,\Omega}  = 0.
\end{align}
Consequently, 
\begin{align*}
	a( \bsig -  \bsig_h, \bsig -  \bsig_h)&= a\inner*{\bsig -  \bsig_h, \bsig - \mathcal S_h(\bPi^{\texttt{BDM}}_h \bsig)}  + a\inner*{\bsig -  \bsig_h, \mathcal S_h(\bPi^{\texttt{BDM}}_h \bsig - \mathcal{A}_h\bsig_h)} 
	\\[1ex]
	&\qquad + a\inner*{\bsig -  \bsig_h, \mathcal S_h(\mathcal{A}_h \bsig_h) - \bsig_h)}
	\\[1ex]
	&\leq a\inner*{\bsig -  \bsig_h, \bsig - \mathcal S_h(\bPi^{\texttt{BDM}}_h \bsig)} + 
	a\inner*{\bsig -  \bsig_h, \mathcal S_h(\mathcal{A}_h \bsig_h) - \bsig_h)},
\end{align*}
and the Cauchy--Schwarz inequality implies that 
\begin{equation}\label{imp}
	a( \bsig -  \bsig_h, \bsig -  \bsig_h)^{\frac{1}{2}}\lesssim  \norm*{\bsig - \mathcal S_h(\bPi^{\texttt{BDM}}_h \bsig)}_{0,\Omega} + \norm*{\bsig_h - \mathcal S_h(\mathcal{A}_h \bsig_h)}_{0,\Omega}.
\end{equation}
The result follows from \eqref{asympSD0} and from the fact that $\inner*{\tr(\bsig - \bsig_h), 1} = 0$. 
\end{proof}

It only remains now to show that the convergence of $p_h:= -\tfrac{1}{d}\tr \bsig_h$ to $p:= -\tfrac{1}{d}\tr \bsig$  is also enhanced by one order in the $L^2$-norm, where  {where $\bsig$ and $\bsig_h$ are the solutions of problems \eqref{mvf} and  \eqref{MDG}, respectively.}
\begin{corollary}\label{L2p}
 Assume that $\set*{\cP_{k+1}(\Omega,\R^d)\cap H^1(\Omega,\R^d),  \cP_{k}(\Omega,\R^d)}$ is a stable Stokes pair on the mesh $\cT_h$.  If  $\bsig\in  H^{r+1}(\cup_j\Omega_j,\bbM)$, with $r\geq 1$, there exists $C>0$ independent of $h$ {and $\mu$, but depending on $\kappa_+/\kappa_-$},   such that 
\begin{align*}
\norm*{p  - p_h}_{0,\Omega}    \leq C  \,  h^{\min\{r+1,k+1\}}
\sum_{j=1}^J   \norm*{ \bsig}_{r+1, \Omega_j},
\end{align*}
for all $\mathtt{a} \geq \mathtt{a}_0$.
\end{corollary}
\begin{proof}
Using the triangle inequality and \eqref{imp} yields
\begin{align}\label{split1}
\begin{split}
	a(\mathcal S_h(\bPi^{\texttt{BDM}}_h \bsig - \mathcal{A}_h\bsig_h), \mathcal S_h(\bPi^{\texttt{BDM}}_h \bsig - \mathcal{A}_h\bsig_h))^{\frac{1}{2}} &\lesssim  
	a( \bsig -  \bsig_h, \bsig -  \bsig_h)^{\frac{1}{2}}
	\\[1ex]
	&+\norm*{\bsig - \mathcal S_h(\bPi^{\texttt{BDM}}_h \bsig)}_{0,\Omega}
	+\norm*{\bsig_h - \mathcal S_h(\mathcal{A}_h\bsig_h)}_{0,\Omega}
	\\[1ex]
	&\lesssim \norm*{\bsig - \mathcal S_h(\bPi^{\texttt{BDM}}_h \bsig)}_{0,\Omega}
	+\norm*{\bsig_h - \mathcal S_h(\mathcal{A}_h\bsig_h)}_{0,\Omega}, 
	\end{split}
\end{align}
which permits us to deduce from \eqref{s-sh3} that 
\begin{align}\label{split2}
\begin{split}
	   \norm*{\kappa^{\frac{1}{2}} \bdiv (\mathcal S_h(\bPi^{\texttt{BDM}}_h \bsig - \mathcal{A}_h\bsig_h))}^2_{0,\Omega} &\leq a\big(\mathcal S_h(\bPi^{\texttt{BDM}}_h \bsig - \mathcal{A}_h\bsig_h), \mathcal S_h(\bPi^{\texttt{BDM}}_h \bsig - \mathcal{A}_h\bsig_h)\big)^{\frac{1}{2}}   
	a\big( \bsig -  \bsig_h, \bsig -  \bsig_h\big)^{\frac{1}{2}}
	 \\[1ex]
	 & \lesssim \Big( \norm*{\bsig - \mathcal S_h(\bPi^{\texttt{BDM}}_h \bsig)}_{0,\Omega}
	+\norm*{\bsig_h - \mathcal S_h(\mathcal{A}_h\bsig_h)}_{0,\Omega}\Big)^2.
	\end{split} 
\end{align}
Consequently, it follows from \eqref{split1}, \eqref{split2} and the coercivity property  \eqref{3.7} that 
\[
  \norm*{\mathcal S_h(\bPi^{\texttt{BDM}}_h \bsig - \mathcal{A}_h\bsig_h)}_{0,\Omega} \lesssim   
 \norm*{\bsig - \mathcal S_h(\bPi^{\texttt{BDM}}_h \bsig)}_{0,\Omega}
	+\norm*{\bsig_h - \mathcal S_h(\mathcal{A}_h\bsig_h)}_{0,\Omega}.
\]
Finally, by virtue of the triangle inequality, 
\begin{align*}
	\sqrt{d}\norm*{p - p_h}_{0,\Omega} \leq \norm*{\bsig - \bsig_h}_{0,\Omega} &\leq \norm*{\bsig - \mathcal S_h(\bPi^{\texttt{BDM}}_h \bsig)}_{0,\Omega} + \norm*{\mathcal S_h(\bPi^{\texttt{BDM}}_h \bsig - \mathcal{A}_h\bsig_h)}_{0,\Omega} + \norm*{\bsig_h - \mathcal S_h(\mathcal{A}_h\bsig_h)}_{0,\Omega}
	\\[1ex]
	&\lesssim \norm*{\bsig - \mathcal S_h(\bPi^{\texttt{BDM}}_h \bsig)}_{0,\Omega}
	+\norm*{\bsig_h - \mathcal S_h(\mathcal{A}_h\bsig_h)}_{0,\Omega},
\end{align*}
and the result follows from \eqref{asympSD0}. 
\end{proof}

\section{Numerical results}\label{sec:results}
We now include a set of numerical examples that illustrate the convergence properties of the proposed DG method, and the usability of the formulation in simulating typical viscous flow in porous media. 
The stabilisation constant depends on the polynomial degree $k\geq 1$, and it is here taken as $\mathtt{a} = \mathtt{a}^* k^2$, where $\mathtt{a}^*$ is specified in each example. All computational tests were conducted using the open source finite element library FEniCS \cite{alnaes}.

\begin{table}[t]
	\centering
{\small	\begin{tabular}{||crc|gg|cccccccccccc||}
	\toprule 
$k$ &{\tt DoF}& $h$ & ${\tt e}_{\vertiii{}}(\bsig)$ &  {\tt r} & ${\tt e}_a(\bsig)$ &  {\tt r} 
& ${\tt e}_{\bdiv_h}(\bsig)$ & {\tt r} 
& ${\tt e}_{\jump{}}(\bsig)$ & {\tt r} 
& ${\tt e}_0(\bu)$ & {\tt r} 
& ${\tt e}_0(\bu^\star)$ & {\tt r} 
& ${\tt e}_0(p)$ & {\tt r} 
 \\
\midrule
\multirow{6}{*}{1}&
   72 & 0.707 & 1.17e+0 & * & 1.08e-01 & * & 8.95e-01 & * & 1.69e-01 & * & 2.32e+2 & * & 2.12e+2 & * & 1.04e-01 & * \\
&   288 & 0.354 & 5.97e-01 & 0.97 & 4.73e-02 & 1.19 & 4.56e-01 & 0.97 & 9.43e-02 & 0.84 & 8.66e+1 & 1.42 & 7.71e+1 & 1.46 & 3.87e-02 & 1.43\\
&  1152 & 0.177 & 2.99e-01 & 1.00 & 2.19e-02 & 1.11 & 2.28e-01 & 1.00 & 4.90e-02 & 0.95 & 3.24e+1 & 1.42 & 2.84e+1 & 1.44 & 1.63e-02 & 1.25\\
&  4608 & 0.088 & 1.49e-01 & 1.00 & 1.07e-02 & 1.04 & 1.14e-01 & 1.00 & 2.48e-02 & 0.98 & 1.17e+1 & 1.47 & 1.01e+1 & 1.50 & 7.64e-03 & 1.09\\
& 18432 & 0.044 & 7.46e-02 & 1.00 & 5.28e-03 & 1.01 & 5.68e-02 & 1.00 & 1.25e-02 & 0.99 & 4.15e+0 & 1.49 & 3.54e+0 & 1.51 & 3.75e-03 & 1.03\\
& 73728 & 0.022 & 3.73e-02 & 1.00 & 2.63e-03 & 1.01 & 2.84e-02 & 1.00 & 6.26e-03 & 1.00 & 1.48e+0 & 1.49 & 1.25e+0 & 1.50 & 1.86e-03 & 1.01\\
\midrule
	\multirow{6}{*}{2} &
   144 & 0.707 & 2.60e-01 & *& 3.41e-02 & * & 2.12e-01 & * & 1.41e-02 & * & 2.23e+01 & * & 2.03e+1 & * & 2.67e-02 & *\\
&   576 & 0.354 & 7.25e-02 & 1.84 & 8.84e-03 & 1.95 & 5.93e-02 & 1.84 & 4.32e-03 & 1.70 & 6.24e+0 & 1.84 & 5.40e+0 & 1.91 & 6.34e-03 & 2.08\\
 & 2304 & 0.177 & 1.87e-02 & 1.96 & 2.25e-03 & 1.97 & 1.53e-02 & 1.96 & 1.16e-03 & 1.90 & 1.48e+0 & 2.07 & 1.24e+0 & 2.12 & 1.58e-03 & 2.00\\
 & 9216 & 0.088 & 4.71e-03 & 1.99 & 5.67e-04 & 1.99 & 3.85e-03 & 1.99 & 2.99e-04 & 1.96 & 3.53e-01 & 2.07 & 2.89e-01 & 2.10 & 3.99e-04 & 1.99\\
 &36864 & 0.044 & 1.18e-03 & 2.00 & 1.42e-04 & 1.99 & 9.63e-04 & 2.00 & 7.56e-05 & 1.98 & 8.55e-02 & 2.04 & 6.92e-02 & 2.06 & 1.00e-04 & 1.99\\
&147456 & 0.022 & 2.96e-04 & 2.00 & 3.56e-05 & 2.00 & 2.41e-04 & 2.00 & 1.90e-05 & 1.99 & 2.10e-02 & 2.03 & 1.69e-02 & 2.03 & 2.52e-05 & 2.00\\
\bottomrule
	\end{tabular}}
	\caption{Error history produced on usual uniform meshes (elements are squares split along one diagonal) and using polynomial degrees $k=1,2$. Error decay and convergence rates for stress, velocity, and pressure approximations.}\label{table1}
\end{table} 

\medskip
\noindent\textbf{Convergence tests.} 
First we compare approximate and closed-form exact solutions for various levels of uniform mesh  refinement. Let us consider the unit square domain $\Omega = (0,1)^2$, the parameters $\mu = 10^{-3}$, $\kappa = 1$, $\mathtt{a}^*=10$, and the following manufactured smooth velocity and pressure 
\[ \bu = \begin{pmatrix}\cos(\pi x)\sin(\pi y) \\ - \sin(\pi x)\cos(\pi y)\end{pmatrix}, \qquad p = \sin(\pi xy),\]
from which the exact Cauchy stress $\bsig_{\mathrm{manuf}}$, forcing term, imposed non-homogeneous velocity $\bG_D$ on $\Gamma_D = \{ 0 \}\times (0,1) \cup (0,1)\times\{1\}$ (left and top sides of the square), and imposed non-homogeneous normal stress $\bG_N$ on $\Gamma_N = \partial\Omega\setminus\Gamma_D$ are constructed. The boundary partition indicates that we use the parameter {$\theta =0$}. 
Errors between exact and approximate solutions computed using \eqref{MDG-nonh} are measured in the following norms
\begin{gather*}
{\tt e}_{\vertiii{}}(\bsig) =  \vertiii{\bsig-\bsig_h}, \quad {\tt e}_a(\bsig) = a(\bsig-\bsig_h,\bsig-\bsig_h)^{\frac12}, \quad {\tt e}_{\bdiv_h}(\bsig) =\|\bdiv_h(\bsig-\bsig_h)\|_{0,\Omega} , \\ 
{\tt e}_{\jump{}}(\bsig) =\norm*{\gamma_\cF^{-\frac{1}{2}} h_F^{-\frac{1}{2}} \jump{\bsig-\bsig_h}}_{0,\cF^*_h} , \quad  {\tt e}_0(\bu) =\|\bu-\bu_h\|_{0,\Omega}  , \quad  {\tt e}_0(\bu^\star) =  \|\bu-\bu^\star_h\|_{0,\Omega},  \quad  {\tt e}_0(p) = \|p-p_h\|_{0,\Omega} ,
\end{gather*}
and the experimental rates of convergence   are computed as 
\[{\tt r}  =\log({\tt e}_{(\cdot)}/\tilde{{\tt e}}_{(\cdot)})[\log(h/\tilde{h})]^{-1}, \]
where ${\tt e},\tilde{{\tt e}}$ denote errors generated on two consecutive  meshes of sizes $h$ and~$\tilde{h}$, respectively. 
Such an error history is displayed, for uniform meshes composed of squares splitted along one diagonal, in Table~\ref{table1}. As anticipated by Theorem~\ref{conv}, a $k$-order of convergence is observed for the Cauchy stress in the norm defined in \eqref{eq:triple-norm}. For sake of clarity we also tabulate the different components of the norm. An agreement with the error estimates from Corollaries~\ref{postV} and \ref{postP} is also observed for the post-processed velocity and pressure in the $L^2$-norms.  For $k=2$  the error bounds are also confirmed using the second velocity post-processing from Lemma~\ref{postV2} (which employs the BDM-mixed pair $\set*{H(\sdiv, \Omega)\cap \cP_{k-1}(\cT_h, \R^2), \cP_{k-2}(\cT_h)}$  for velocity and the auxiliary field). We show the decay of the error in the $H(\sdiv)$-norm, but we recall that the divergence of the discrete velocity is zero to machine precision. Note that for $k=1$  the second velocity post-processing  needs to be done with the BDM-mixed pair  $\set*{H(\sdiv, \Omega)\cap \cP_{1}(\cT_h, \R^2), \cP_{0}(\cT_h)}$. Note also that, for the type of meshes considered in Table~\ref{table1},  the additional order of convergence for the deviatoric stress and for the pressure in the $L^2$-norm are not achieved. However, if we consider instead special partitions using, for example, \emph{simplicial barycentric trisected} -- Hsieh--Clough--Tocher meshes -- or \emph{twice quadrisected crisscrossed} meshes (as required by the Scott--Vogelius element \cite[Section 55.3.1]{ern1}), the additional order shown in Theorem~\ref{L2} and Corollary~\ref{L2p} is clearly obtained (see the sixth and rightmost columns of  Table~\ref{table2}). 


\begin{table}[t]
	\centering
{\small	\begin{tabular}{||crc|gg|mmccccccccmm||}
	\toprule 
$k$ &{\tt DoF}& $h$ & ${\tt e}_{\vertiii{}}(\bsig)$ &  {\tt r} & ${\tt e}_a(\bsig)$ &  {\tt r} 
& ${\tt e}_{\bdiv_h}(\bsig)$ & {\tt r} 
& ${\tt e}_{\jump{}}(\bsig)$ & {\tt r} 
& ${\tt e}_0(\bu)$ & {\tt r} 
& ${\tt e}_0(\bu^\star)$ & {\tt r} 
& ${\tt e}_0(p)$ & {\tt r} 
 \\
\midrule
\multirow{6}{*}{1}&
   144 & 0.500 & 7.15e-01 & * & 3.04e-02 & * & 5.89e-01 & * & 9.55e-02 & * & 1.85e+2 & * & 1.48e+2 & * & 2.89e-02 & * \\
&   576 & 0.250 & 3.54e-01 & 1.02 & 7.24e-03 & 2.07 & 2.91e-01 & 1.02 & 5.56e-02 & 0.78 & 7.27e+1 & 1.35 & 5.85e+1 & 1.34 & 7.08e-03 & 2.03\\
&  2304 & 0.125 & 1.75e-01 & 1.01 & 1.76e-03 & 2.04 & 1.44e-01 & 1.02 & 2.95e-02 & 0.91 & 2.82e+1 & 1.37 & 2.25e+1 & 1.38 & 1.77e-03 & 2.00\\
&  9216 & 0.062 & 8.71e-02 & 1.01 & 4.35e-04 & 2.02 & 7.14e-02 & 1.01 & 1.52e-02 & 0.96 & 1.12e+1 & 1.33 & 8.77e+0 & 1.36 & 4.44e-04 & 2.00\\
& 36864 & 0.031 & 4.34e-02 & 1.00 & 1.08e-04 & 2.01 & 3.56e-02 & 1.01 & 7.71e-03 & 0.98 & 4.67e+0 & 1.26 & 3.56e+0 & 1.30 & 1.11e-04 & 2.00\\
&147456 & 0.016 & 2.17e-02 & 1.00 & 2.69e-05 & 2.00 & 1.78e-02 & 1.00 & 3.88e-03 & 0.99 & 2.07e+0 & 1.18 & 1.53e+0 & 1.22 & 2.78e-05 & 2.00\\
\midrule
\multirow{6}{*}{2} &
  288 & 0.500 & 9.90e-02 & * & 3.08e-03 & * & 9.06e-02 & * & 5.32e-03 & * & 1.20e+01 & * & 8.54e+00 & * & 3.60e-03 & * \\
&  1152 & 0.250 & 2.45e-02 & 2.01 & 3.85e-04 & 3.00 & 2.27e-02 & 2.00 & 1.47e-03 & 1.85 & 2.75e+00 & 2.12 & 2.06e+00 & 2.05 & 4.52e-04 & 2.99\\
&  4608 & 0.125 & 6.04e-03 & 2.02 & 4.88e-05 & 2.98 & 5.60e-03 & 2.02 & 3.91e-04 & 1.91 & 5.95e-01 & 2.21 & 4.72e-01 & 2.13 & 5.68e-05 & 2.99\\
& 18432 & 0.062 & 1.50e-03 & 2.01 & 6.16e-06 & 2.99 & 1.40e-03 & 2.01 & 1.01e-04 & 1.95 & 1.36e-01 & 2.13 & 1.12e-01 & 2.07 & 7.16e-06 & 2.99\\
& 73728 & 0.031 & 3.75e-04 & 2.00 & 7.75e-07 & 2.99 & 3.48e-04 & 2.00 & 2.57e-05 & 1.98 & 3.23e-02 & 2.07 & 2.74e-02 & 2.03 & 8.99e-07 & 2.99\\
&294912 & 0.016 & 9.36e-05 & 2.00 & 9.71e-08 & 3.00 & 8.71e-05 & 2.00 & 6.47e-06 & 1.99 & 7.87e-03 & 2.04 & 6.77e-03 & 2.02 & 1.13e-07 & 3.00\\
\midrule  
	\multirow{6}{*}{3} &
  480 & 0.500 & 1.55e-02 & * & 3.02e-04 & * & 1.48e-02 & * & 4.02e-04 & * & 2.02e+00 & * & 1.36e+00 & * & 4.34e-04 & * \\
&  1920 & 0.250 & 1.75e-03 & 3.15 & 1.97e-05 & 3.94 & 1.68e-03 & 3.14 & 4.82e-05 & 3.06 & 1.98e-01 & 3.35 & 1.23e-01 & 3.46 & 2.55e-05 & 4.09\\
&  7680 & 0.125 & 2.14e-04 & 3.03 & 1.29e-06 & 3.94 & 2.06e-04 & 3.03 & 6.42e-06 & 2.91 & 2.07e-02 & 3.25 & 1.14e-02 & 3.44 & 1.62e-06 & 3.98\\
& 30720 & 0.062 & 2.65e-05 & 3.01 & 8.27e-08 & 3.96 & 2.56e-05 & 3.01 & 8.36e-07 & 2.94 & 2.30e-03 & 3.17 & 1.09e-03 & 3.38 & 1.03e-07 & 3.98\\
&122880 & 0.031 & 3.31e-06 & 3.00 & 5.29e-09 & 3.97 & 3.19e-06 & 3.00 & 1.07e-07 & 2.97 & 2.66e-04 & 3.11 & 1.11e-04 & 3.30 & 6.54e-09 & 3.97\\
&491520 & 0.016 & 4.15e-07 & 2.99 & 2.86e-10 & 3.99 & 3.99e-07 & 3.00 & 1.35e-08 & 2.98 & 3.19e-05 & 3.06 & 1.19e-05 & 3.22 & 3.15e-10 & 3.99 \\

\bottomrule
	\end{tabular}}
	\caption{Error history produced on meshes with twice quadrisected crisscrossed elements, and using polynomial degrees $k=1,2,3$. Error decay and convergence rates for stress, velocity, and pressure approximations.}\label{table2}
\end{table} 

We remark that the  limit case of $\kappa=0$ cannot be studied with the present formulation. Moreover, while for the other limit case of $\mu =0$ the method  converges optimally in all the stress norms as well as in the pressure post-processing, we cannot use \eqref{eq:post-cont} to recover velocity (and the second post-process depends on the first one). Nevertheless, we point out that the method still converges optimally for stress and all post-processed fields, for very low values of viscosity and/or permeability, and for very high values of permeability (confirmed through tests over several orders of magnitude).

\begin{figure}[t!]
  \begin{center}
    \includegraphics[width=0.4\textwidth]{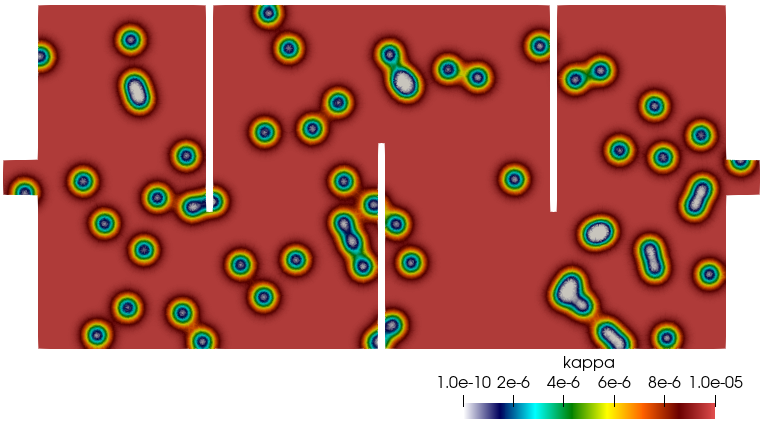}
       \raisebox{6mm}{\includegraphics[width=0.17\textwidth]{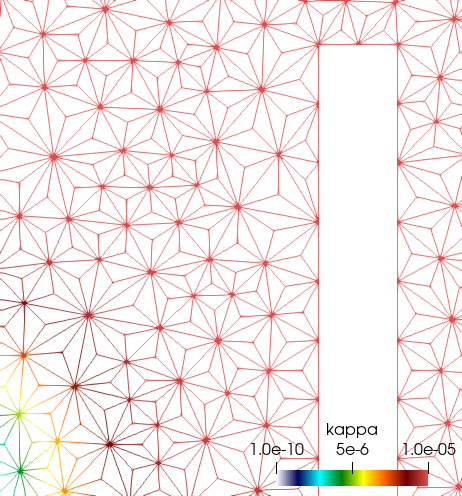}}
    \includegraphics[width=0.4\textwidth]{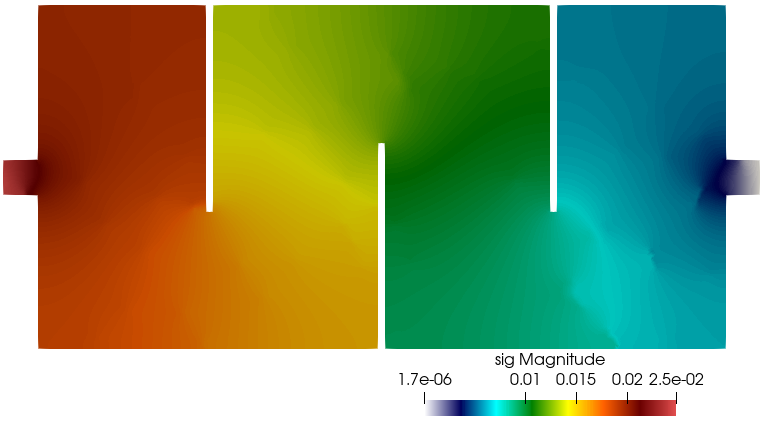}
\includegraphics[width=0.4\textwidth]{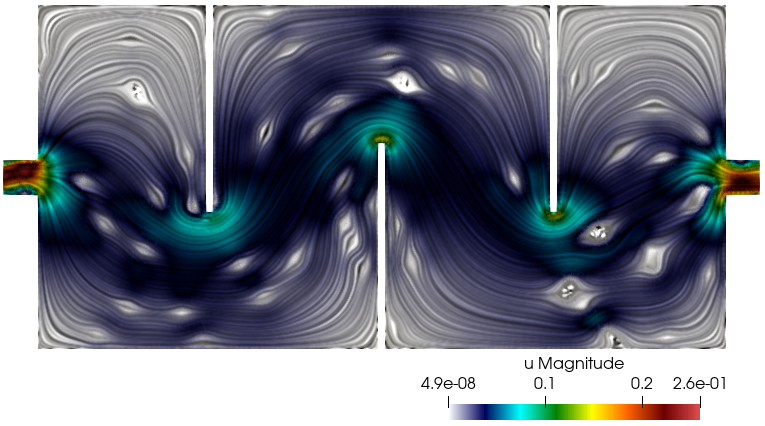}\qquad
\includegraphics[width=0.4\textwidth]{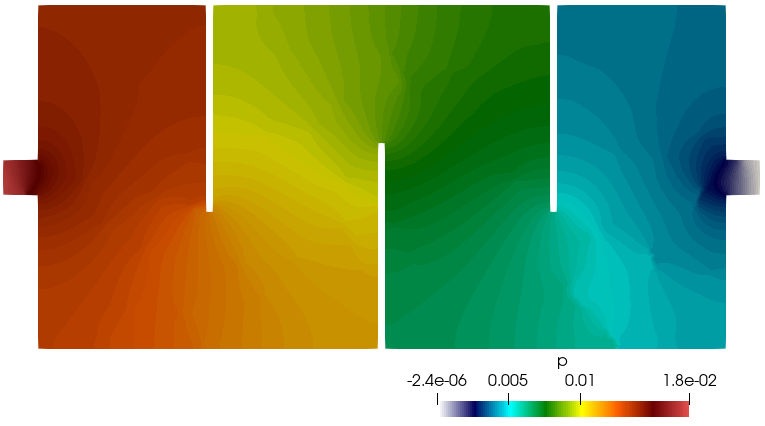}
  \end{center}

  \vspace{-4mm}  
  \caption{Flow on a maze-shaped domain. Heterogeneous permeability distribution, zoom to visualize the simplicial barycentric trisected mesh, Cauchy stress magnitude, first post-process of velocity, and post-processed pressure.}\label{fig:maze}
\end{figure}

\medskip
\noindent\textbf{2D maze and channel flow with heterogeneous permeability.}
Now we focus on two different geometries. For the first 2D example we consider a maze-shaped geometry of length 2.2 and height 1 (in adimensional units). We use an unstructured simplicial barycentric trisected grid with 66006 elements, representing, for $k=1$, a total of 594054 degrees of freedom. The rightmost segment of the boundary is the outlet ($\Gamma_N$), where the outflow condition $\bsig\bn = \boldsymbol{0}$ is imposed. The remainder of the boundary is $\Gamma_D$, split between the leftmost segment of the boundary (the inlet, where we impose the parabolic profile $\bG_{\mathrm{in}} =(100(y-0.45)(0.55-y),0)^{\tt t}$) and the walls where $\bG_{\mathrm{wall}} = \boldsymbol{0}$. The external force is zero, the viscosity is $\mu = 10^{-6}$, and the permeability $\kappa(\bx)$ {is characterized by a non-homogeneous field taking the value $\kappa_{+} = 10^{-5}$ everywhere on the domain, except on 60 small disks distributed randomly, where the permeability smoothly goes down to the much smaller value $\kappa_{-} = 10^{-10}$: 
\[ \kappa(\bx) = \max(\kappa_{+}(1-\sum_{i=1}^{60}\exp[-800\{(x-q_x(i))^2+(y-q_y(i))^2\} ]),\kappa_{-} ),\] where $(q_x(i),q_y(i))$ denote the coordinates of the randomly located points}. 
The stabilization parameter is  $\mathtt{a}^*=15$. Figure~\ref{fig:maze} displays the permeability field, the Cauchy stress magnitude, the line integral convolutions of post-processed velocity, and the distribution of post-processed pressure, where we observe that the expected symmetry of the flow is disrupted by the heterogeneous permeability. These results indicate well-resolved approximations.

In the second 2D example we follow \cite{kan} and compute channel flow solutions and take the permeability distribution from the Cartesian SPE10 benchmark  dataset for reservoir simulations / model 2 (we choose layer 45, corresponding to a fluvial fan pattern with channeling. See also, e.g., \cite{aarnes07,christie01}). 
The permeability data has a very large contrast: the minimum and maximum values are $\kappa_{-} = 1.3\cdot10^{-18}\,[\text{m}^2]$ and $\kappa_{+} = 2.0\cdot10^{-11}\,[\text{m}^2]$, and it is 
projected onto a twice quadrisected crisscrossed mesh discretizing the rectangular domain $\Omega = (0,6.096)\,[\text{m}]\times(0,3.048)$\,[m].  The remaining parameters are $\mathtt{a}^*=25$ and $\mu = 10^{-6}\,[\text{KPa}\cdot\text{s}]$.  On the inlet (the left segment of the boundary) 
we impose the inflow velocity $\bu = (10,0)^{\tt t}\,$[m/s], on the outlet (the right vertical segment) we set zero normal stress, and on the horizontal walls we impose no-slip conditions. 

The flow patterns are shown in Figure \ref{fig:spe}, and the qualitative behaviour coincides with the expected filtration mechanisms observed elsewhere (see, e.g., \cite{krot}). The magnitude of velocity in its second post-process $|\bu^*_h|$, and the stress magnitude are plotted in logarithmic scale. 

\begin{figure}[t!]
  \begin{center}
    \includegraphics[width=0.325\textwidth]{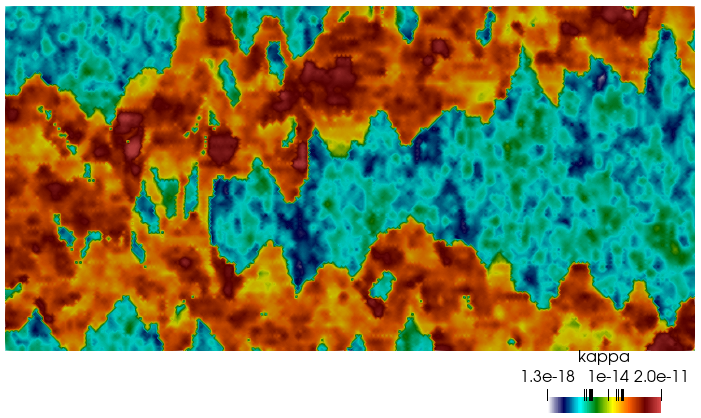}
    \includegraphics[width=0.325\textwidth]{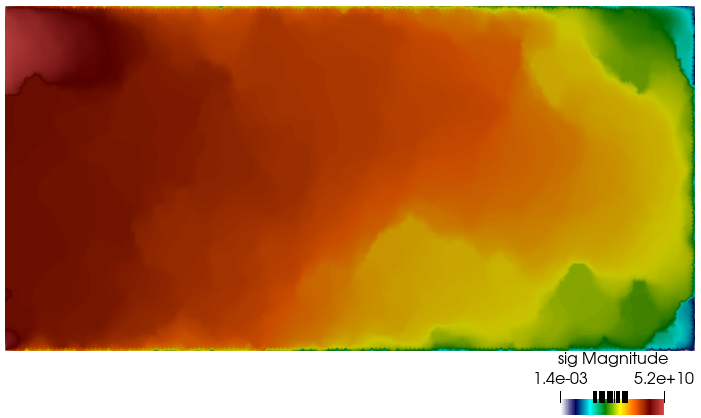}
\includegraphics[width=0.325\textwidth]{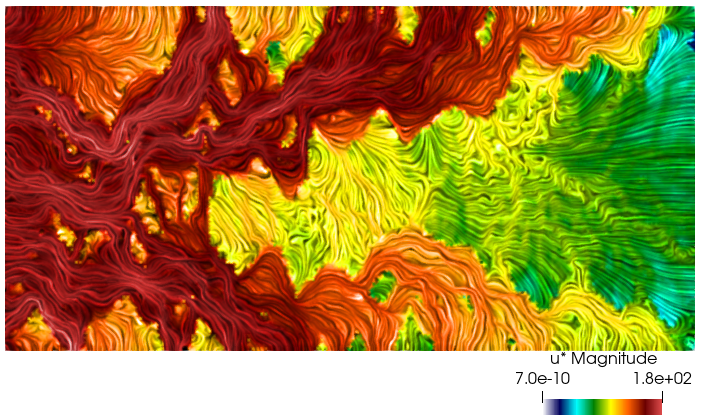}
  \end{center}

  \vspace{-4mm}
  \caption{Channel flow with permeability from the SPE10--layer 45 benchmark data, and using a twice quadrisected crisscrossed mesh. Heterogeneous permeability distribution in log scale, Cauchy stress magnitude  in log scale, and line integral convolution of second post-process of velocity in log scale.}\label{fig:spe}
\end{figure}

\medskip
\noindent\textbf{3D flow through porous media.} 
To close this section, in this test we use a cylindrical geometry of radius 0.2\,[m] and height  0.7\,[m], discretized into an unstructured simplicial barycentric quadrisected mesh of 48756 tetrahedra. We set the viscosity to the value of water $\mu = 8.9\cdot 10^{-4}\,[\mathrm{Pa}\cdot\mathrm{s}]$ and again use inlet (the face $z=0$), wall (the surface $r=0.2$\,[m]), and outlet (the face $z=0.7$\,[m]) type of boundary conditions mimicking a channel flow with homogeneous right-hand side in the momentum balance $\bF= \boldsymbol{0}$. The inlet velocity is $\bG_{\mathrm{in}} = (0,0,1)^{\tt t}\,[\mathrm{m/s}]$. A synthetic permeability field is generated based on a porosity of approximately 0.4. For this we generate a normalized uniform random distribution, apply a Gaussian filter with standard deviation of 2, and then use a quantile of 40\%. The resulting field is rescaled between  $\kappa_{-} = 10^{-12}\,[\text{m}^2]$ and $\kappa_{+} = 10^{-4}\,[\text{m}^2]$. Figure~\ref{fig:3D} presents the permeability and contour iso-surfaces of velocity magnitude, which indicate the formation of wormhole-like structures with higher velocity. 

\begin{figure}[t!]
  \begin{center}
\includegraphics[width=0.325\textwidth]{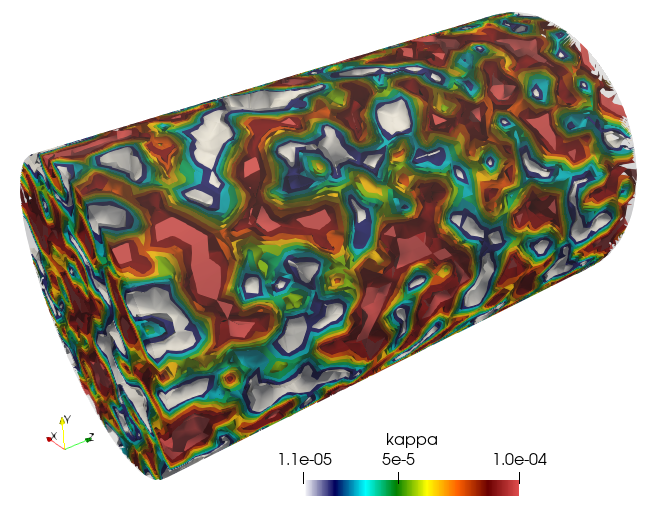}
\includegraphics[width=0.325\textwidth]{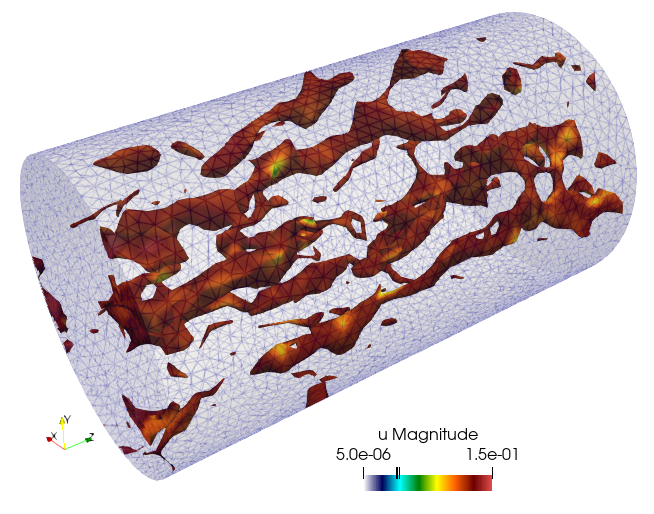}
\includegraphics[width=0.325\textwidth]{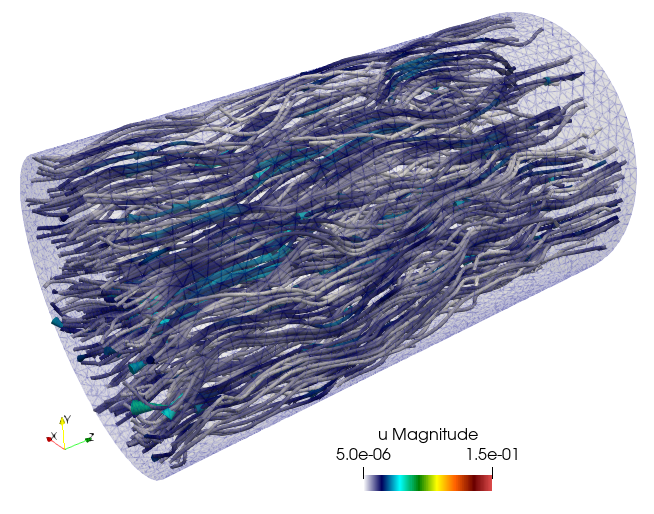}
  \end{center}

  \vspace{-4mm}
  \caption{Channel flow with synthetic permeability. The mesh is of simplicial barycentric quadrisected type. Heterogeneous permeability distribution,  contour iso-surfaces of velocity magnitude in log scale, and velocity streamlines.}\label{fig:3D}
\end{figure}




\end{document}